\documentclass[12pt]{amsart}
\usepackage{mathdots}
\usepackage[margin=1.25in]{geometry}

\usepackage{amsmath}
\usepackage{amsfonts}
\usepackage{amssymb}
\usepackage{amscd}
\usepackage{graphicx}
\usepackage[abbrev,alphabetic]{amsrefs}
\RequirePackage[dvipsnames,usenames]{color}
\usepackage{soul,xcolor}
\setstcolor{red}
\usepackage{stmaryrd}
\usepackage{multicol}

\newcommand\hidden[1]{} 
\usepackage{mathtools}
\newtagform{hidden}[\hidden]{}{}

\usepackage{hyperref}

\usepackage{amsthm}
\usepackage{comment}
\usepackage[all,cmtip]{xy}
\usepackage{tikz-cd}
\usetikzlibrary{cd}

\makeatletter
\def\@tocline#1#2#3#4#5#6#7{\relax
  \ifnum #1>\c@tocdepth 
  \else
    \par \addpenalty\@secpenalty\addvspace{#2}%
    \begingroup \hyphenpenalty\@M
    \@ifempty{#4}{%
      \@tempdima\csname r@tocindent\number#1\endcsname\relax
    }{%
      \@tempdima#4\relax
    }%
    \parindent\z@ \leftskip#3\relax \advance\leftskip\@tempdima\relax
    \rightskip\@pnumwidth plus4em \parfillskip-\@pnumwidth
    #5\leavevmode\hskip-\@tempdima
      \ifcase #1
       \or\or \hskip 1em \or \hskip 2em \else \hskip 3em \fi%
      #6\nobreak\relax
    \hfill\hbox to\@pnumwidth{\@tocpagenum{#7}}\par
    \nobreak
    \endgroup
  \fi}
\makeatother

\hypersetup{
bookmarks,
bookmarksdepth=3,
bookmarksopen,
bookmarksnumbered,
pdfstartview=FitH,
colorlinks,backref,hyperindex,
linkcolor=Sepia,
anchorcolor=BurntOrange,
citecolor=MidnightBlue,
citecolor=OliveGreen,
filecolor=BlueViolet,
menucolor=Yellow,
urlcolor=OliveGreen
}

\newsavebox{\pullback}
\sbox\pullback{%
\begin{tikzpicture}%
\draw (0,0) -- (1ex,0ex);%
\draw (1ex,0ex) -- (1ex,1ex);%
\end{tikzpicture}}

\newsavebox{\pullbackdl}
\sbox\pullbackdl{%
\begin{tikzpicture}%
\draw (-1ex,0ex) -- (0ex,0ex);%
\draw (0ex,-1ex) -- (0ex,0ex);%
\end{tikzpicture}}

\newsavebox{\pushoutdr}
\sbox\pushoutdr{%
\begin{tikzpicture}%
\draw (-1ex,-1ex) -- (-1ex,0ex);%
\draw (-1ex,0ex) -- (0ex,0ex);%
\end{tikzpicture}}

\newcommand{\stacksproj}[1]{{\cite[Tag~{#1}]{stacks-project}}}

\newcommand{\mustata}{Musta{\c{t}}{\u{a}}}

\newcommand{\cExt}{\mathcal{E}xt}

\newcommand{\bC}{\mathbb{C}}
\newcommand{\bD}{\mathbb{D}}

\newcommand{\bH}{\mathbb{H}}

\newcommand{\bP}{\mathbb{P}}
\newcommand{\bQ}{\mathbb{Q}}

\newcommand{\bZ}{\mathbb{Z}}

\newcommand{\cC}{\mathcal{C}}
\newcommand{\cD}{\mathcal{D}}

\newcommand{\cF}{\mathcal{F}}
\newcommand{\cG}{\mathcal{G}}
\newcommand{\cH}{\mathcal{H}}

\newcommand{\cL}{\mathcal{L}}

\newcommand{\cM}{\mathcal{M}}
\newcommand{\cO}{\mathcal{O}}

\newcommand{\m}{\mathfrak{m}}
\newcommand{\p}{\mathfrak{p}}

\DeclareMathOperator{\Gr}{Gr}
\DeclareMathOperator{\DR}{DR}
\DeclareMathOperator{\GrDR}{Gr\,DR}

\DeclareMathOperator{\DDB}{\underline{\Omega}}
\DeclareMathOperator{\DO}{D\underline{\Omega}}
\DeclareMathOperator{\IO}{I\hspace{0.07em}\underline{\Omega}}
\DeclareMathOperator{\IC}{IC}

\DeclareMathOperator{\Supp}{Supp}
\DeclareMathOperator{\Spec}{Spec}

\DeclareMathOperator{\codim}{codim}
\DeclareMathOperator{\Hom}{Hom}

\theoremstyle{plain}
\newtheorem{theorem}{Theorem}[section]

\newtheorem{proposition}[theorem]{Proposition}
\newtheorem{lemma}[theorem]{Lemma}
\newtheorem{corollary}[theorem]{Corollary}

\newtheorem{claim}[theorem]{Claim}
\newtheorem*{claim*}{Claim}

\newtheorem{theoremIntro}{Theorem}

\newtheorem{lemmaIntro}[theoremIntro]{Lemma}

\newtheorem{propositionIntro}[theoremIntro]{Proposition}

\newtheorem{corollaryIntro}[theoremIntro]{Corollary}

\theoremstyle{definition}
\newtheorem{definition}[theorem]{Definition}

\newtheorem{setting}[theorem]{Setting}

\newtheorem*{setup*}{Setup}

\theoremstyle{remark}
\newtheorem{remark}[theorem]{Remark}

\numberwithin{equation}{theorem}



\newif\ifshowColoursAndTodoes
\showColoursAndTodoestrue

\ifshowColoursAndTodoes
\def\todo#1{\textcolor{Mahogany}%
{\footnotesize\newline{\color{Mahogany}\fbox{\parbox{\textwidth-15pt}{\textbf{todo: } #1}}}\newline}}
\def\commentbox#1{\textcolor{Mahogany}%
{\footnotesize\newline{\color{Mahogany}\fbox{\parbox{\textwidth-15pt}{\textbf{comment: } #1}}}\newline}}

\else
\def\todo#1{}
\def\commentbox#1{}
\renewcommand{\st}[1]{}
\colorlet{red}{black!100} 
\colorlet{teal}{black!100} 
\colorlet{brown}{black!100} 
\colorlet{blue}{black!100} 
\colorlet{magenta}{black!100} 
\colorlet{purple}{black!100} 
\colorlet{cyan}{black!100} 
\fi

\newcommand{\kdot}{{{\,\begin{picture}(1,1)(-1,-2)\circle*{2}\end{picture}\,}}}

\title[Inversion of adjunction]{Inversion of adjunction for higher rational singularities}

\author{Tatsuro Kawakami}
\address{Graduate School of Mathematical Sciences, University of Tokyo, 3-8-1 Komaba,
Meguro-ku, Tokyo 153-8914, Japan} 
\email{tatsurokawakami0@gmail.com}

\author{Jakub Witaszek} 
\address{Northwestern University, Department of Mathematics, Lunt Hall, 2033 Sheridan Road, Evanston, IL 60208, USA}
\email{jakub.witaszek@northwestern.edu}

\begin{document}

\begin{abstract}
We prove inversion of adjunction for higher rational singularities. 
\end{abstract}

\maketitle

\setcounter{tocdepth}{2}
\tableofcontents

\section{Introduction}

Recent years have seen a surge of interest in the study of singularities coming from Hodge theory, especially in the context of Saito's theory of Hodge modules. In this article, we focus on two such classes of singularities: higher Du Bois and higher rational.

These classes of singularities can be defined in terms of the filtered complex $\DDB^\kdot_X$, called the \emph{Deligne--Du Bois complex}, which features prominently in the theory of mixed Hodge structures for singular varieties. The complexes $\DDB^i_X \in D^b_{\rm coh}(X)$ agree with $\Omega^i_X$ when $X$ is smooth. In general, following \cite{SVV} we say that
\begin{align*} \label{intro:def-hypersurface}
\text{$X$ is \emph{$m$-Du Bois} if }& \DDB^i_X \cong \Omega^{[i]}_X \text{ for } 0 \leq i \leq m \text{ and } {\rm codim}({\rm Sing}(X)) \geq 2m+1,\\
    \text{$X$ is \emph{$m$-rational} if }& Rf_*\Omega^i_Y(\log E) \cong \Omega^{[i]}_X \text{ for } 0 \leq i \leq m \text{ and } {\rm codim}({\rm Sing}(X)) > 2m+1. \nonumber
\end{align*}
Here $f \colon Y \to X$ is a \emph{strong resolution}, that is, a log resolution of singularities which is an isomorphism over the smooth locus. Also $E$ is the exceptional locus of $f$ and $\Omega^{[i]}_X \coloneqq (\Omega^i_X)^{**}$. 
One can check that $m$-Du Bois and $m$-rational singularities agree with the standard Du Bois and rational singularities, respectively, when $m=0$.

\begin{remark} $m$-Du Bois singularities are strongly connected to $m$-log canonical singularities introduced by {\mustata} and Popa (\cite{MP19}). Specifically, in \cite{MOPW23} and \cite{JKSY22}, it was proven that $m$-log canonical singularities are equivalent to $m$-Du Bois singularities for hypersurfaces (see \cites{Mustata-Popa22,CDM23k-rational-complete-intersection} for the case of complete intersections). 

The class of $m$-rational singularities was introduced by Friedman and Laza (\cite{Friedman-Laza2,Friedman-Laza1}), motivated, in part, by their study of unobstructedness of deformation theory of singular complex Calabi-Yau varieties (see also the work of Kerr and Laza \cite{KL24}). Both $m$-Du Bois and $m$-rational singularities have been extensively studied in the recent years, for example in the work of Shen--Venkatesh--Vo (\cite{SVV}), Popa--Shen--Vo (\cite{popa2024injectivityvanishingdubois}), Park--Popa (\cite{PP24}), Kawakami--Witaszek (\cite{KW24}),  {\mustata}--Popa (\cite{mustata2024krationalkduboislocal}), Chen--Dirks--{\mustata}  (\cite{CDM23k-rational-complete-intersection}), and many others. 
\end{remark}

It is a general expectation in algebraic geometry that if a Cartier divisor $D$ has mild singularities, then the ambient space $X$ should also have mild singularities along $D$. This phenomenon is known as \emph{the inversion of adjunction} and has been proven for many classes of singularities such as rational (\cite{elkik78}) or Du Bois (\cite{KovacsSchwede16}). Establishing inversion of adjunction for higher singularities has remained an important unresolved problem in the field. The main result of this article is a proof of this property for $m$-rational singularities.

\begin{theoremIntro}[{Theorem \ref{thm:inversion-of-adjunction}}] \label{intro:main-thm}
Let $X$ be a normal connected $d$-dimensional variety   defined over $\bC$ and let $D$ be a normal Cartier prime divisor. Assume that $D$ is $m$-rational. Then $X$ is $m$-rational along $D$.
\end{theoremIntro} \label{thm:main-intro}
\noindent In fact, our theorem is a bit stronger (see Definition \ref{defintro:-predef} and Theorem \ref{thm:inversion-of-adjunction}). 
Also note that this result was previously known for complete intersections via a comparison with $m$-log canonical singularities and the minimal exponent (see the work of Chen \cite{chen24} for the strongest results in this direction).

\begin{corollaryIntro} \label{corintro:fibres}
Let $X$ be a normal connected $d$-dimensional variety defined over $\bC$. Let $g \colon X \to C$ be a proper morphism over $\bC$ with connected fibres to a smooth connected curve $C$. Fix an integer $m\geq 0$, pick a closed point $s \in C$, and set $X_s$ to be the fibre of $g$ over $s$.  Assume that $X_s$ is normal and has $m$-rational singularities. Then all but a finite number of fibres of $g$ have $m$-rational singularities.
\end{corollaryIntro}
\begin{remark}
One can also ask whether inversion of adjunction holds for $m$-Du Bois singularities. We plan to address this question in a future article. 
\end{remark}

\subsection{Equivalent definitions of higher singularities} \label{ss:definition-k-rational}

In order to prove Theorem \ref{intro:main-thm}, we need to reformulate the definition of $m$-Du Bois and $m$-rational singularities purely in terms of local cohomology. We start by reviewing different variants of higher singularities. Let $X$ be a normal connected $d$-dimensional variety defined over $\bC$.

\begin{definition} \label{defintro:-predef}
Following \cite{SVV}, we say that
    \begin{itemize}
    \item $X$ is \emph{pre-$m$-Du Bois} if $\cH^j(\DDB^i_X) =0$ for all integers $0 \leq i \leq m$ and $j >0$,
    \item $X$ is \emph{pre-$m$-rational} if $\cH^j(\DO^{i}_X)=0$ for all integers $0 \leq i \leq m$ and $j>0$.
\end{itemize}
\end{definition}
\noindent Here $\DO^i_X \coloneqq R\Hom(\DDB^{d-i}_X, \omega^\kdot_X[-d])$. If we assume that ${\rm codim}({\rm Sing}(X)) > m$, then
\[
\DO^i_X = Rf_*\Omega^i_Y(\log E)
\]
for any strong resolution $f \colon Y \to X$ (Proposition \ref{prop:dualDDB-explicit}). Set $\Omega^i_{X,h} \coloneqq \cH^0(\DDB^i_X)$.

\begin{propositionIntro}[{Proposition \ref{prop:lc-def-ofkDuBois}}]\label{propintro:DDB-def} Let $X$ be a normal connected $d$-dimensional variety defined over $\bC$. Then 
$X$ is pre-$m$-Du Bois if and only if the following natural map is injective
\[
H^j_\m(\Omega^i_{X,h}) \to H^j_\m(\DDB^i_X)
\]
for all maximal ideals $\m$ and integers $i,j \geq 0$ such that $i+j \leq d$ and $i \leq m$.
\end{propositionIntro}
\noindent The above proposition is related to the recent very deep theorem of Kovács \cite{kovács2025complexesdifferentialformssingularities}. A variant of the above result was already 
discussed in the first version of \cite{KW24}.

\begin{propositionIntro}[{Corollary \ref{cor:lc-def-ofkRational}}] \label{propintro:IO-def} Let $X$ be a normal connected $d$-dimensional variety defined over $\bC$ and let $f \colon Y \to X$ be a log resolution of $X$ with exceptional divisor $E$. Then $X$ is pre-$m$-rational if and only if the following natural map is injective:
\[
H^j_\m(\Omega^i_{X,h}) \to H^j_\m(Rf_*\Omega^i_Y(\log E))
\]
for all maximal ideals $\m$, and for all integers $i,j \geq 0$ such that $i+j \leq d$ and $i \leq m$.
\end{propositionIntro}

\noindent We emphasise that in Proposition \ref{propintro:IO-def}, we do not need to assume that $f$ is a strong resolution of singularities. Moreover, in this proposition we can replace $Rf_*\Omega^i_Y(\log E)$ by $\DO^{i}_X$ or the intersection complex differential forms $\IO^i_X$ (see Definition \ref{def:io}). The key case of $\IO^i_X$ is deduced from Proposition \ref{propintro:DDB-def} and \cite[Theorem 10.5]{popa2024injectivityvanishingdubois}. In turn, this yields Proposition \ref{propintro:IO-def} by invoking the simplicity of the intersection complex $\IC_X$. Finally, the connection between the study of $\IO^i_X$ and $\DO^i_X$ is established by \cite{DOR25} (see also \cite{PP24} and \cite{popa2024injectivityvanishingdubois}). \\

Finally, let us point out that in the proof of Theorem \ref{intro:main-thm} we used the following small lemma. Since we found its statement somewhat surprising, we decided to highlight it in the introduction.

\begin{lemmaIntro} \label{lem:small-introduction}
Let $X$ be a normal connected variety of dimension $d$ defined over $\bC$. Fix an integer $m>0$. Suppose that $X$ has rational and pre-$(m-1)$-Du Bois singularities. Then
\[
\Supp\, \cExt^{-2}(\Omega^{[m]}_X, \omega^\kdot_X)
\]
is a finite set. In particular, by local duality (\ref{eq:local-duality}), the set
\[
\{ \text{closed points } x \in X \mid H^2_x(\Omega^{[m]}_X) \neq 0 \}
\]
is finite.
\end{lemmaIntro}

\subsection{Outline}
For the convenience of the reader, we aim to keep our article largely self-contained. To this end, we blackbox all essential properties of Hodge modules via Proposition \ref{prop:JakubsFavouriteFunctor}, which summarises the properties of the key $\GrDR$ functor. We derive all results\footnote{We also blackbox the fact that the derived category of mixed Hodge modules admits a $6$-functor formalism and behaves reasonably well under restriction to a general hyperplane; see Corollary \ref{cor:restrictionDDB-and-IO}.} from this proposition and purely topological properties of the category ${\rm Perv}_{\rm cons}(X,\bC)$ of perverse constructible $\bC$-sheaves. This -- less D-module heavy -- approach more closely reflects the perspective inherent in positive and mixed characteristic (Remark \ref{remark:BL}). Although the article could be significantly shortened by omitting most of the preliminaries we decided to set everything up carefully for the sake of future references, especially in the context of our endeavour to build the theory of higher $F$-injective and higher $F$-rational singularities in positive characteristic. 

We also emphasise that our article is strongly built on and motivated by the work of Shen–Venkatesh–Vo (\cite{SVV}), Popa–Shen–Vo (\cite{popa2024injectivityvanishingdubois}), Popa--Park (\cite{PP24}),  Chen--Dirks--{\mustata} (\cite{CDM23k-rational-complete-intersection}), and many other developments mentioned earlier.

\subsection{Acknowlegements}
The authors thank Bhargav Bhatt, Brad Dirks, Hyunsuk Kim, Sándor Kovács, Mircea Musta{\c{t}}{\u{a}}, Mihnea Popa, Christian Schnell, Duc Vo, Shou Yoshikawa, Zhi Zheng for valuable conversations related to the content of the paper.
Kawakami was supported by JSPS KAKENHI Grant number JP24K16897.
Witaszek was supported by NSF research grants DMS-2101897 and DMS-2401360.

\section{Preliminaries} \label{ss:preliminaries}

In this paper we only work with $k=\bC$. We start by summarising some basic notation and definitions.
\begin{enumerate}
    \item A \emph{variety} over a field $k$ is a scheme, separated and of finite type over $k$. The dimension of a variety $X$ is the maximum over the dimensions of $\cO_{X,x}$ for every closed point $x \in X$.
    \item A \emph{log resolution} $f \colon Y \to X$ of a normal variety $X$ is a projective birational morphism such that $Y$ is regular and the exceptional divisor $E$ is simple normal crossing (see \cite[Tag 0BI9]{stacks-project}).
    \item A \emph{strong resolution} $f \colon Y \to X$ of a normal variety $X$ is a log resolution such that $f$ is an isomorphism over the smooth locus $X_{\rm sm}$.
    \item Given a $d$-dimensional variety $X$ over a field $k$, we define the \emph{dualising complex} $\omega^\kdot_X \in D^b_{\rm coh}(X)$ by the formula $\omega^\kdot_X \coloneqq f^!(k)$, where $f \colon X \to \Spec k$ is the natural projection (cf.\ \cite[Tag 0A85]{stacks-project}). It is of cohomological amplitude $[-d,0]$. Set $\omega_X \coloneqq \cH^{-d}(\omega^\kdot_X)$. When $X$ is normal, $\omega_X$ agrees with the reflexivisation of $\det(\Omega^1_X)$.
    \item $\DDB^i_X \in D^b_{\rm coh}(X)$ denote the $i$-th Du Bois complex of $X$. When $X$ is smooth, $\DDB^i_X = \Omega^i_X$. In general, it is defined by the formula
    \[
    \DDB^i_X \coloneqq R\epsilon_{\kdot,*}\Omega^i_{X_\kdot},
    \]
    where $\epsilon_\kdot \colon  X_{\kdot} \to X$ is a hyperresolution (see \cite[Chapter 5]{GNPP} or \cite[Chapter 7.3]{Peter-Steenbrink(Book)} for details). We set $\Omega^i_{X,h} \coloneqq \cH^0(\DDB^i_X)$. By construction $\Omega^i_{X,h}$ is torsion-free.  Equivalently, we can define $\Omega^i_{X,h}$ as the $h$-sheafification of the sheaf $\Omega^i$ on the $h$-site over $X$. Similarly, $\DDB^i_X$ is the derived $h$-sheafification of $\Omega^i$. Finally, we often denote $\DDB^0_X$ and $\DDB^d_X$ by $\underline{\cO}_X$ and $\underline{\omega}_X$, respectively. Note that the latter object agrees with the Grauert--Riemenschneider sheaf:
    \[
    \underline{\omega}_X = f_*\omega_Y,
    \]
    with $f \colon Y \to X$ a resolution of a normal variety $X$ (\cite[Remark 3.4]{popa2024injectivityvanishingdubois}).
    \item Let $(R,\m)$ be a local Noetherian ring and let $K \in D^b_{\rm coh}(R)$. Then
    \begin{equation} \label{eq:local-duality}
    \cExt^{-i}_R(K, \omega^\kdot_R)^{\wedge} \cong H^i_\m(K)^{\vee}.
    \end{equation}
    This identity is called \emph{local duality} (\stacksproj{0AAK}).
    Here $(-)^{\wedge}$ denotes $\m$-adic derived completion and $(-)^{\vee} = \Hom_R(-,E)$ denotes \emph{Matlis duality} with $E$ being the injective hull of $R/\m$ of $R$. We refer to \stacksproj{08XG} for details (see also \cite[Section 2.4]{tanaka2024quasifesplittingsquasifregularity}).
\end{enumerate}
\vspace{0.5em}
We shall repeatedly use the following result.
\begin{lemma} \label{lem:basic-ses}
Let $X$ be a smooth variety defined over an algebraically closed field $k$ and let $D+E$ be a reduced simple normal crossing divisor with $D$ being prime. Then there exist the following short exact sequences:
\begin{align*}
&0 \to \Omega^i_X(\log E) \to \Omega^i_X(\log D+E) \to \Omega^{i-1}_D(\log E\cap D) \to 0\\
&0 \to \Omega^i_X(\log D+E)(-D) \to \Omega^i_X(\log E) \to \Omega^i_D(\log E \cap D) \to 0.
\end{align*}
\end{lemma}
\begin{proof}
See \cite[Properties 2.3]{EV92}.
\end{proof}

\begin{remark} \label{remark:DDB-filtered}
We briefly discuss the Deligne Du-Bois complex $\DDB^\kdot_X$ on a $d$-dimensional variety $X$ over $\bC$. It is constructed via derived $h$-sheafification from the de Rham complex $\Omega^\kdot$ with stupid (Hodge) filtration. It is an element of the filtered derived category and as such it comes with fixed maps
\[
\DDB^\kdot_X = F^0\DDB^\kdot_X \xleftarrow{\gamma_1} F^1\DDB^\kdot_X \xleftarrow{\gamma_2} \ldots \xleftarrow{\gamma_{d}} F^d \DDB^\kdot_X
\]
in $D^b(X,\bC)$. Denote the composition \[
\Gamma_i := \gamma_1 \circ \ldots \circ \gamma_i \colon F^i\DDB^\kdot_X \to \DDB^\kdot_X.
\]

By the 9-lemma (see \cite[Tag 05R0]{stacks-project} or \cite[Proposition 1.1.11]{BBDG18}) 
 applied to the diagram:
\[
\begin{tikzcd}
F^{i+1}\DDB^\kdot_X \ar{d}{\gamma_i} \ar{r}{{\Gamma}_{i+1}} & \DDB^\kdot_X \ar{d}{=} \\
F^{i}\DDB^\kdot_X \ar{r}{{\Gamma}_{i}}  &\DDB^\kdot_X
\end{tikzcd}
\]
we get the following diagram in which each column and row is exact
\begin{equation} \label{eq:DDBfilteredexact}
\begin{tikzcd}
\DDB^i_X[-i-1] \ar{d} \ar{r} & 0 \ar{d} \ar{r} & \DDB^i_X[-i] \ar{d}{\theta_i}  \ar{r}{+1} & \hphantom{a} \\
F^{i+1}\DDB^\kdot_X \ar{d}{\gamma_i} \ar{r}{{\Gamma}_{i+1}} & \DDB^\kdot_X \ar{d}{=} \ar{r}{\Theta_{i}} & \DDB^{\leq i}_X  \ar{d} \ar{r}{+1} & \hphantom{a} \\
F^{i}\DDB^\kdot_X \ar{r}{{\Gamma}_{i}} \ar{d}{+1} &\DDB^\kdot_X \ar{r}{\Theta_{i-1}} \ar{d}{+1} & \DDB^{\leq i-1}_X \ar{r}{+1} \ar{d}{+1} & \hphantom{a}\\
\hphantom{a} & \hphantom{a} & \hphantom{a}
\end{tikzcd}
\end{equation}
The elements $\DDB^{\leq i}_X$ of $D^b(X,\bC)$ are uniquely defined (up to a non-unique isomorphism) from this diagram. The term $\DDB^i_X$ agrees with derived $h$-sheafification of $\Omega^i$, and so it naturally comes from $D^b_{\rm coh}(X)$. We only labelled the maps $\theta_i$ and $\Theta_i$ in the above diagram as the choice of other maps will be irrelevant to us. We emphasise that these maps are not canonical; we fix them throughout this article.
\end{remark}

\subsection{Review of $t$-structure}

Let $\cD$ be a triangulated category. Let us briefly recall the concept of a $t$-structure; we refer to \cite{FM06} for an introduction to triangulated categories and to \cite{HTT} and \cite{BBDG18} for a more detailed discussion on $t$-structures. 

A $t$-structure on $\cD$ consists of a pair of full subcategories $\cD^{\leq 0}$ and $\cD^{\geq 0}$ satisfying various properties, set up in a way so that the heart of the $t$-structure:
\[
\cD^{\heartsuit} \coloneqq \cD^{\leq 0} \cap \cD^{\geq 0}
\]
is an abelian category\footnote{A basic example of a $t$-structure is given on $D(X)$, the derived category of abelian sheaves on $X$, in which we take $\cD^{\leq 0}$ and $\cD^{\geq 0}$ to denote complexes of cohomological amplitude $[-\infty, 0]$ and $[0, \infty]$, respectively (see \cite[Example 8.1.2]{HTT})} (see \cite[Theorem 8.1.9(a)]{HTT}). We define $\cD^{\leq i} \coloneqq \cD^{\leq 0}[-i]$ and $\cD^{\geq i} \coloneqq \cD^{\geq 0}[-i]$. There exist truncation functors (\cite[Proposition 8.1.4]{HTT}):
\[
{}^t\tau^{\leq i} \colon \cD \to \cD^{\leq i} \quad \text{ and } \quad {}^t\tau^{\geq i} \colon \cD \to \cD^{\geq i}
\]
defined as adjoints to natural inclusions $\cD^{\leq i} \to \cD$ and $\cD^{\geq i} \to \cD$, respectively. Therewith, for $K \in \cD$ we can define
\[
{}^t\cH^i(-) \colon \cD \to \cD^{\heartsuit} \quad \text{ given by} \quad  {}^t\cH^i(-) \coloneqq  ({}^t\tau^{\leq i} \circ {}^t\tau^{\geq i}(-))[i] = ({}^t\tau^{\geq i}\circ {}^t\tau^{\leq i}(-))[i],
\]
 where the last equality follows from \cite[Proposition 8.18 (iii)]{HTT}.
\begin{remark}
All the $t$-structures we are going to work with will be \emph{non-degenerate} in the sense that $\bigcap_{i} \cD^{\geq i} = 0$ and $\bigcap_{i} \cD^{\leq i} = 0$. Under this assumption, the system of functors ${}^t\cH^i(-)$ is conservative. In particular, $K \in \cD^{\heartsuit}$ is equivalent to ${}^t\cH^i(K) = 0$ for $i \neq 0$ (see \cite[Proposition 1.3.7]{BBDG18}).  
\end{remark}
Any exact triangle $A \to B \to C \xrightarrow{+1}$ in $\cD$ induces a long exact sequence (\cite[Proposition 8.1.11]{HTT}):
\begin{equation} \label{eq:les-of-t-cohomology}
\ldots \to {}^t\cH^i(A) \to {}^t\cH^i(B) \to {}^t\cH^i(C) \to {}^t\cH^{i+1}(A) \to \ldots
\end{equation}
of objects of the abelian category $\cD^{\heartsuit}$. In particular, we see that if $A, C \in \cD^{\heartsuit}$, then $B \in \cD^{\heartsuit}$, as well.

\begin{remark} \label{rem:t-ker-coker}
Let $B, C \in \cD^{\heartsuit}$ and let $f \colon B \to C$ be any map. Extend this map to an exact triangle
\[
A \to B \to C \xrightarrow{+1}
\]
in $\cD$. By the long exact sequence (\ref{eq:les-of-t-cohomology}), we see that:
\begin{align*}
{\rm ker}(B \to C) &= {}^t\cH^0(A) \in \cD^{\heartsuit}\\  
{\rm coker}(B \to C) &= {}^t\cH^1(A) \in \cD^{\heartsuit}.
\end{align*}
In particular, $B \to C$ is an injection in $\cD^{\heartsuit}$ if and only if ${}^t{\mathcal{H}}^0(A)=0$ and a surjection if and only if ${}^t{\mathcal{H}}^1(A) = 0$.
\end{remark}

Given an exact functor $F \colon \cD_1 \to \cD_2$ between triangulated categories with $t$-structures, we define
\[
{}^tF \colon \cD_1^{\heartsuit} \to \cD_2^{\heartsuit} \quad \text{  by the formula } \quad {}^tF(-) \coloneqq {}^t\cH^0(F(-)).
\]
Moreover, we say that $F \colon \cD_1 \to \cD_2$  is 
\begin{itemize}
    \item \emph{left $t$-exact} if $F(\cD^{\geq 0}_1) \subseteq \cD^{\geq 0}_2$, and 
    \item \emph{right $t$-exact} if $F(\cD^{\leq 0}_1) \subseteq \cD^{\leq 0}_2$. 
    \item \emph{$t$-exact} if $F$ is left $t$-exact and right $t$-exact.
\end{itemize}

\subsection{Review of perverse coherent $t$-structure}
In this subsection, $X$ is a variety over $\bC$. For a more detailed review of the perverse coherent $t$-structure, we refer to \cite[Section 3.1]{BMPSTWW2} and \cite[Section 4.1]{BBLSZ25}.

The bounded derived category of coherent sheaves $D^b_{\rm coh}(X)$ admits a standard $t$-structure with the heart being exactly ${\rm Shv}_{\rm coh}(\cO_X)$. In what follows we define a different $t$-structure, called \emph{perverse coherent $t$-structure}, whose goal is to capture the Cohen-Macaulayness property.

\begin{definition} \label{def:perv-cons-t}
For $K \in D^b_{\rm coh}(X)$, we define:
\begin{enumerate}
    \item $K \in {}^pD^{\leq 0}(X)$ if and only if one of the following equivalent conditions holds (see \cite[Lemma 3.2]{BMPSTWW2}):
        \begin{itemize} 
            \item $K_x \in D^{\leq -\dim\overline{\{x\}}}$ for all points $x \in X$
            \item $R\Gamma_x(K_x) \in D^{\leq -\dim\overline{\{x\}}}$ for all points $x \in X$.
        \end{itemize}
    \item $K \in {}^pD^{\geq 0}(X)$ if and only if $R\Gamma_x(K_x) \in D^{\geq -\dim\overline{\{x\}}}$ for all  points $x \in X$.
\end{enumerate}
\end{definition}
\begin{remark} \label{rem:perv-coh-maximal-only}
In fact the condition that $R\Gamma_x(K_x) \in D^{\leq -\dim\overline{\{x\}}}$ or  $R\Gamma_x(K_x) \in D^{\leq -\dim\overline{\{x\}}}$ can be checked at closed points only. Specifically (see \cite[Remark 3.5(a)]{BMPSTWW2}):
\begin{enumerate}
    \item $K \in {}^pD^{\leq 0}(X)$ if and only if $R\Gamma_x(K_x) \in D^{\leq 0}$ for all closed points $x \in X$.
    \item $K \in {}^pD^{\geq 0}(X)$ if and only if $R\Gamma_x(K_x) \in D^{\geq 0}$ for all closed points $x \in X$.
\end{enumerate}
\end{remark}

\begin{remark}
    Note that \cite[Lemma 3.4]{BMPSTWW2} and subsequent results such as those contained in \cite[Remark 3.5(a)]{BMPSTWW2} assume that $X$ is equidimensional. This is explained therein to guarantee $\omega^\kdot_{X,x}[-\dim \overline{ \{x\}}]$ to be normalised for $\cO_{X,x}$. However, this normalisation is always guaranteed by \cite[Tag 0AWF]{stacks-project} and the fact that a dimension function is unique up to shift. Thus, we do not need to assume equidimensionality in this section.
\end{remark}

The above definition endows $D^b_{\rm coh}(X)$ with a $t$-structure (called \emph{perverse coherent $t$-structure}) whose heart:
\[
{\rm Perv}_{\rm coh}(X) \coloneqq {}^pD^{\geq 0} \cap {}^pD^{\leq 0} \subseteq D^b_{\rm coh}(X)
\]
is an abelian category (cf.\ \cite[Definition 4.5 and Remark  4.11]{BBLSZ25}). Note that a coherent sheaf $M \in {\rm Shv}_{\rm coh}(\cO_X) \subseteq D^b_{\rm coh}(X)$ with irreducible support of dimension $e$ satisfies the property that $M[e]$ is perverse if and only if $M$ is Cohen-Macaulay!
\begin{remark}
Although the above definitions may seem excessively formal, they allow for discussing Cohen-Macaulayness and local cohomology in a much more efficient manner. For example, given two Cohen-Macaulay coherent sheaves $M, N$ with support equal to $X$ we have that
\[
M[d] \to N[d]
\]
is an injection in the abelian category ${\rm Perv}_{\rm coh}(X)$ if and only if
\[
H^d_\m(M) \to H^d_\m(N)
\]
is injective for every maximal ideal $\m$. This easily follows by considering an exact triangle:
\[
K \to M[d] \to N[d] \xrightarrow{+1}
\]
and noting that $M[d] \to N[d]$ is injective if and only if $K \in {}^pD^{>0}$ (see Remark \ref{rem:t-ker-coker}) if and only if $H_\m^{>0}(K)=0$ for every maximal ideal $\m$ (see Remark \ref{rem:perv-coh-maximal-only}) if and only if $H_\m^0(M[d]) \to H_\m^0(N[d])$ is injective.
\end{remark}

The triangulated category $D^b_{\rm coh}(X)$ admits a Grothendieck duality functor
\[
\bD_X \colon D^b_{\rm coh}(X) \to D^b_{\rm coh}(X) \quad \text{ given by } \quad \bD_X(-) = R\Hom_{\cO_X}(-, \omega^\kdot_X),
\]
where $\omega^\kdot_X$ is the dualising complex. Under this functor, the perverse $t$-structure is dual to the standard $t$-structure. Specifically (see \cite[Lemma 3.4]{BMPSTWW2}): 
\begin{align} \label{eq:perv-coh-dual}
K \in {}^pD^{\leq 0}(X)\quad  &\text{ if and only if }\quad  \bD_X(K) \in D^{\geq 0}(X), \text{ and } \\
K \in {}^pD^{\geq 0}(X)\quad &\text{ if and only if }\quad  \bD_X(K) \in D^{\leq 0}(X). \nonumber
\end{align}
By the above, the dualising complex $\omega^\kdot_X$ is an example of a perverse coherent sheaf as $\bD_X(\omega^\kdot_X) \cong \cO_X$.

\subsection{Review of perverse constructible $t$-structure} \label{ss:perverse-constructible}
In this subsection, $X$ is a variety over $\bC$. We will only work with the middle perverse $t$-structure. 
When dealing with constructible sheaves, all functors will denote derived functors unless otherwise stated; for example, we always write $f_*$ instead of $Rf_*$. For a more detailed review, we refer to \cite[Section 8.1.2]{HTT}.

We say that a $\bQ_{X^{\rm an}}$-modules $F$ is \emph{algebraically constructible} if there exists a stratification $X = \bigsqcup_{\alpha \in A} X_{\alpha}$ such that $F|_{X^{\rm an}_{\alpha}}$ is locally constant for every $\alpha \in A$ (see \cite[Definition 4.5.6]{HTT}). We denote by $D^b_{\rm cons}(X, \bQ)$ the full subcategory of $D^b_{\rm cons}(X^{\rm an}, \bQ)$ consisting of bounded complexes of $\bQ_{X^{\rm an}}$-modules whose cohomology groups are algebraically constructible. By abuse of notation we shall denote $\bQ_{X^{\rm an}}$ by $\bQ_X$.

\begin{remark} \label{remark:exact-triangles}
Let $j \colon U \to X$ be an inclusion of an open subscheme and let $i \colon Z \to X$ be the closed embedding of the complement of $U$. We shall repeatedly use the following standard exact triangles for $K \in D^b_{\rm cons}(X,\bQ)$:
\begin{align*}
&j_!j^* K \to K \to i_*i^*K \xrightarrow{+1} \\
&i_*i^! K \to K \to j_*j^*K \xrightarrow{+1}.
\end{align*}
\end{remark}

\begin{definition} \label{def:perverse-cons-t-structure}
Let $K \in D^b_{\rm cons}(X, \bQ)$ and let $X = \bigsqcup_{\alpha \in A} X_{\alpha}$ be any stratification consisting of connected strata such that $i_{\alpha}^*K$ and $i_{\alpha}^!K$ have locally constant cohomology sheaves where $i_{\alpha} \colon X_{\alpha} \to X$ are the natural inclusions. We define (\cite[Proposition 8.1.22]{HTT}):
\begin{enumerate}
    \item $K \in {}^pD^{\leq 0}(X)$ if and only if $i_{\alpha}^*K \in D^{\leq - {\rm dim}(X_{\alpha})}$ for every $\alpha \in A$.
    \item $K \in {}^pD^{\geq 0}(X)$ if and only if $i_{\alpha}^!K \in D^{\geq - {\rm dim}(X_{\alpha})}$ for every $\alpha \in A$.
\end{enumerate}
\end{definition}
\noindent The definition is independent of the choice of the stratification (\cite[Proposition 8.1.24]{HTT}). Moreover, it endows $D^b_{\rm cons}(X,\bQ)$ with a $t$-structure (called \emph{perverse $t$-structure}) whose heart:
\[
{\rm Perv}_{\rm cons}(X,\bQ) \coloneqq {}^pD^{\geq 0} \cap {}^pD^{\leq 0} \subseteq D^b_{\rm cons}(X,\bQ)
\]
is an abelian category (\cite[Theorem 8.1.27]{HTT}). 

\begin{remark} An example of a perverse constructible sheaf is given by $\bQ_X[d]$ when $X$ is smooth 
 (\cite[Proposition 8.1.31]{HTT}). In general:
\begin{equation} \label{eq:trivial-q-leq0}
\bQ_X[d] \in {}^pD^{\leq 0} 
\end{equation}
by Definition \ref{def:perverse-cons-t-structure}(1) as $i_{\alpha}^*$ is exact and $i^*_{\alpha}\bQ_X[d] = \bQ_{X_{\alpha}}[d]$ lives in degree $-d$.
\end{remark}

A key property of the perverse constructible $t$-structure is that it is self-dual. 
\begin{definition} For the orientation complex $\omega^\kdot_X \in D^b_{\rm cons}(X,\bQ)$ we define Verdier duality (\cite[Definition C.2.14 and C.2.15]{HTT}):
\[
\bD_X \colon D^b_{\rm cons}(X,\bQ) \to D^b_{\rm cons}(X,\bQ) \ \text{ given by }\  \bD_X(-) = R\Hom_{\bQ_X}(-, \omega^\kdot_X).
\]
\end{definition}
\noindent Then $\bD_X \circ \bD_X \cong {\rm id}$ and for $K \in D^b_{\rm cons}(X,\bQ)$ we have (\cite[Proposition 8.1.33]{HTT}):
\begin{equation} \label{eq:cons-self-dual}
K \in {}^pD^{\leq 0} \iff \bD_X(K) \in {}^pD^{\geq 0}.
\end{equation}


One of the most important examples of a perverse sheaf is the intersection complex $\IC_X \in {\rm Perv}_{\rm cons}(X,\bQ)$. From now on, following \cite{HTT}, we assume that $X$ is irreducible.

\begin{definition} \label{def:IC} Pick a smooth open dense subscheme $j \colon U \hookrightarrow X$ and set (\cite[Definition 8.2.2 and 8.2.13]{HTT}):
\[
\IC_X \coloneqq {\rm im}\big({}^pj_!(\bQ_U[d]) \to {}^pj_*(\bQ_U[d])\big).
\]
This object is well-defined, that is,  independent of the choice of $j$.
\end{definition}
Above we used the existence of the natural map $j_!(\bQ_U[d]) \to j_*(\bQ_U[d])$ in $D^b_{\rm cons}(X,\bQ)$. The intersection complex agrees with $\bQ_X[d]$ when $X$ is smooth. 

\begin{lemma} The key properties of $\IC_X$ are as follows\footnote{the map in (\ref{IC2}) is formed by applying ${}^p\cH^0(-)$ to $j_!(\bQ_U[d]) \to \bQ_X[d] \to j_*(\bQ_U[d])$}:
\begin{align}
&\text{ $\IC_X$ is self-dual: } \bD_X(\IC_X) = \IC_X, \label{IC1} \\
&\text{ the following map is a surjection of perverse sheaves: }\! {}^p\cH^0(\bQ_X[d]) \to \IC_X, \label{IC2} \\
&\text{ $\IC_X$ is simple\footnotemark\ when $X$ is irreducible,} \label{IC3} \\
&\IC_X \in D^{<0}_{\rm cons}(X,\bQ), \text{ that is, } \cH^{i}(\IC_X)=0 \text{ for $i \geq 0$}, \text{provided $\dim X>0$}. \label{IC4}
\end{align}
\end{lemma}
\footnotetext{it admits no non-trivial sub-objects or quotients}
\begin{proof}
As for (\ref{IC1}) see \cite[Theorem 8.2.14]{HTT}, for (\ref{IC3}) see \cite[Corollary 8.2.10 and Theorem 8.2.24]{HTT}, and for (\ref{IC4}) see \cite[Proposition 8.2.11]{HTT}. In order to see (\ref{IC2}), we need to prove that $(\star)$ in the following factorisation is a surjection of perverse sheaves
\[
{}^pj_!(\bQ_U[d]) \xrightarrow{(\star)} {}^p\cH^0(\bQ_X[d]) \to {}^pj_*(\bQ_U[d]).
\]
This can be deduced from the exact triangle:
\[
j_!(\bQ_U[d]) \to \bQ_X[d] \to i_*\bQ_Z[d] \xrightarrow{+1},
\]
where $i \colon Z \to X$ is the closed embedding of the complement of $U$. Indeed, since $i_*$ is $t$-exact (\cite[Corollary 8.1.44(ii)]{HTT}) and $\bQ_Z[\dim Z] \in {}^pD^{\leq 0}_{\rm cons}(Z,\bQ)$ by (\ref{eq:trivial-q-leq0}), we get that $i_*\bQ_Z[d] \in {}^pD^{< 0}_{\rm cons}(X,\bQ)$. Hence the surjectivity of $(\star)$ follows by the long exact sequence of perverse cohomology (\ref{eq:les-of-t-cohomology}).
\end{proof}


\subsection{Hodge modules} \label{ss:HM}
Let $X$ be a reduced connected variety of dimension $d$ defined over $\bC$ and let $u \colon X \hookrightarrow W$ be a closed embedding into a smooth variety $W$ of dimension $d_W$. It is beyond this article to provide an in-depth introduction to Saito's theory of Hodge modules. We refer the reader to the excellent survey \cite{SchnellSurvey19} on this topic. For convenience of the reader, we blackbox all the properties of Hodge modules in this article and deduce all the proofs directly from the properties of the category of constructible sheaves ${\rm Perv}_{\rm cons}(X,\bQ)$ reviewed above. From this perspective, the reader might think of a Hodge module as a perverse constructible sheaf with \emph{some additional data}. Formally, a mixed Hodge module $\cM$ consists of the following:
\begin{enumerate}
    \item[(a)] a perverse constructible sheaf $K \in {\rm Perv}_{\rm cons}(X, \bQ)$, and
    \item[(b)] a left $D$-module\footnote{Formally speaking, by a $D$-module on $X$ we mean a $D$-module on the smooth ambient space $W$ with support contained in $X$. A category of such objects is independent of the choice of $W$ and the embedding $u \colon X \to W$ (cf.\ \cite[Proposition 6.1]{Schnell16}).} $M$ together with increasing filtrations $F_\kdot$ and  $W_\kdot$,
\end{enumerate}
that satisfy various intricate properties. In particular, $K$ determines the $D$-module $M$ via the Riemann-Hilbert correspondence, but not the filtrations $F_\kdot$ and $W_\kdot$.

We write $\cM = (K, M, F_\kdot, W_\kdot)$. We denote the underlying perverse constructible sheaf $K$ of $\cM$ by ${\rm rat}(\cM)$. We say that a non-zero mixed Hodge module $\cM$ is \emph{pure} if $\Gr^W_\kdot(\cM)$ is non-trivial in one degree only. For technical reasons, one also introduces \emph{Tate's twist} of a mixed Hodge module $\cM$:
\[
\cM(k) \coloneqq (K, M, F_{\kdot -k}, W_{\kdot + 2k}).
\]

The derived category of mixed Hodge modules is denoted by $D^b_{\rm MHM}(X)$. The formation of ${\rm rat}(-)$ extends to an exact functor ${\rm rat} \colon D^b_{\rm MHM}(X) \to D^b_{\rm cons}(X,\bQ)$. Similarly, Tate's twist extends to an exact functor of derived categories as well. 

In order to understand proofs in this article, the reader need only  be aware of the following facts about mixed Hodge modules.
\begin{enumerate}
    \item The derived category $D^b_{\rm MHM}(X)$ of mixed Hodge modules admits a $t$-structure which agrees with the standard $t$-structure on the derived category of $D$-modules and with the perverse $t$-structure on $D^b_{\rm cons}(X, \bQ)$ from Subsection \ref{ss:perverse-constructible}. Specifically, the following functor is $t$-exact
    \[
    {\rm rat} \colon D^b_{\rm MHM}(X) \to D^b_{\rm cons}(X, \bQ).
    \]
    In particular\footnote{Since in practice, one considers only one $t$-structure on $D^b_{\rm MHM}(X)$, we drop superscript $t$ or $p$ and denote the Hodge module cohomology just by $\cH^i$}, ${\rm rat}(\cH^i(\cM)) = {}^p\cH^i({\rm rat}(\cM))$ for every integer $i$.
    \item The above functor ${\rm rat}$ is conservative. In particular, for $\cM \in D^b_{\rm MHM}(X)$, we have that $\cM = 0$ if and only if ${\rm rat}(\cM)=0$. 
    \item $D^b_{\rm MHM}(X)$ admits a $6$-functor formalism $(f_*, f_!, f^*, f^!, \Hom(-,-), - \otimes -)$ which agrees with that on $D^b_{\rm cons}(X, \bQ)$; for example, we have $f_* \circ {\rm rat} = {\rm rat} \circ f_*$ (cf.\ \cite[Theorem 2.4, Proposition 2.5]{Park24}). In this context, the exact triangles from Remark \ref{remark:exact-triangles} extend to the derived category of mixed Hodge modules. Moreover, there is a duality functor
    \[
    \bD_X \colon D^b_{\rm MHM}(X) \to D^b_{\rm MHM}(X)
    \]
    which agrees with Verdier duality $\bD_X$ on $D^b_{\rm cons}(X,\bQ)$, namely: \[
    \bD_X \circ {\rm rat} =  {\rm rat} \circ \bD_X.
    \]
    Following our convention for constructible sheaves, all the functors on mixed Hodge modules are automatically derived in this article.
    \item There exists a \emph{trivial} object $\bQ^H_X \in D^b_{\rm MHM}(X)$ such that ${\rm rat}(\bQ_X^H) = \bQ_X$. When $X$ is smooth:
    \begin{equation} \label{eq:dual-trivial-HM}
    \bD_X(\bQ^H_X[d]) \cong \bQ^H_X[d](d).
    \end{equation}
\end{enumerate}
\noindent In the smooth case, $\bQ^H_X[d]$ is a Hodge module with underlying perverse constructible sheaf $\bQ_X[d]$, the left $D$-module $\cO_X$, and trivial filtrations $F_\kdot$ and $W_{\kdot}$, specifically: 
\[
F_{<0} \cO_X = W_{<d} \cO_X = 0 \quad \text{ and } \quad F_{\geq 0} \cO_X = W_{\geq d} \cO_X = \cO_X.
\]
In general, we define $\bQ^H_X \coloneqq a^*{\bQ^H_{\rm pt}}$ where $a \colon X \to \Spec \bC$ is the natural map.

\begin{remark}
All Hodge modules in this article are constructed from $\bQ^H_X$ via the $6$-functor formalism. In particular, when $X$ is irreducible we can define the intersection complex Hodge module:
\[
\IC^H_X \coloneqq {\rm im}(\cH^0(j_!\bQ^H_{U}[d]) \to \cH^0(j_*\bQ^H_{U}[d]))
\]
so that ${\rm rat}(\IC^H_X) = \IC_X \in D^b_{\rm cons}(X,\bC)$. It is self dual up to Tate twist (\cite[4.5.13]{Saito90}):
\begin{equation} \label{eq:dual-IC}
\bD_X(\IC^H_X) \cong \IC^H_X(d).
\end{equation}
\end{remark}

All the proofs in this article are formally deduced from the above properties by ways of Proposition \ref{prop:JakubsFavouriteFunctor}, Lemma \ref{lem:grdr-boundary} and Proposition \ref{prop:restriction-hodge-module} below.

\begin{remark} \label{rem:linsys-rest}
Before moving forward, we need the following small observation. Let $X$ be a projective variety over $\bC$ and let $\cL$ be a base-point-free line bundle on $X$. Then there exists an embedding $u \colon X \hookrightarrow W$ and a line bundle $\cL_W$ such that the restriction map
$H^0(W, \cL_W) \to  H^0(X, 
\cL)$
is surjective. 

To see this, let $f \colon X \to \bP^m$ be a map induced by the linear system $|\cL|$ so that $f^*\cO_{\bP^m}(1)=\cL$. Since $X$ and $\bP^m$ are projective, the morphism $f$ is $H$-projective (see \stacksproj{0C4Q} and \stacksproj{087S}). This means that $X$ admits an embedding $u \colon X \hookrightarrow {\mathbb{P}}^n_{\bP^m} =: W$ over $f \colon X \to \bP^m$. This embedding together with $\cL_W \coloneqq \pi^*\cO_{\bP^m}(1)$  for the projection $\pi \colon W \to \bP^m$ satisfy the required properties. 
\end{remark}

\subsection{$\GrDR$ functor}
We continue with the notation from Subsection \ref{ss:HM}. One of the most important functors in the theory of Hodge modules is the $\GrDR$ functor:
\[
\GrDR \colon D^b_{\rm MHM}(X) \to D^b_{\rm coh}(X) \qquad \cM \mapsto \GrDR(\cM).
\]
\begin{remark} \label{rem:grdr-in-action}
Specifically, let $\cM$ be a mixed Hodge module corresponding to $(M,F_\kdot)$ a filtered $D$-module on $W$ with support on $X$. Set
\begin{equation} \label{eq:defGRDR}
\Gr^F_k \DR(\cM) \coloneqq \Gr^F_k M \to \Gr^F_{k+1} M \otimes_{\cO_W} \Omega^1_W \to \cdots \to  \Gr^F_{k+{d_W}} M \otimes_{\cO_W} \Omega_W^{d_W},  
\end{equation}
with the complex living in degrees $[-d_W,0]$. 
Then
\[
\GrDR(\cM) \coloneqq \bigoplus_k \Gr^F_k \DR(\cM).
\]
This construction automatically extends to $D^b_{\rm MHM}(X)$. We drop the superscript $F$ as we never explicitly consider any other filtration such as the weight filtration $W_\kdot$. 

\end{remark}
We list properties of this functor below. 
\begin{proposition} \label{prop:JakubsFavouriteFunctor}
Let $X$ be a reduced connected variety over $\bC$. In what follows, $\cM \in D^b_{\rm MHM}(X)$. Then $\GrDR \colon D^b_{\rm MHM}(X) \to D^b_{\rm coh}(X)$ satisfies the following properties:
\begin{enumerate}
\item it is stable under Tate twist, for every $k \in \bZ$:
\[
\GrDR(\cM(k)) = \GrDR(\cM),
\]
\item it commutes with $Rf_*$ for every projective map $f \colon Y \to X$:
\[
Rf_* \GrDR(\cM) = \GrDR(f_*\cM),
\]
\item it commutes with Hodge module duality $\bD_X$ on the source and the Grothendieck duality $\bD_X(-) \coloneqq R\Hom(-, \omega^\kdot_X)$ on the target:
\[
\bD_X \GrDR(\cM) = \GrDR(\bD_X \cM),
\]
\item it is right $t$-exact for the Hodge module $t$-structure on the source and the standard $t$-structure on the target:
\[
\GrDR(D^{\leq 0}) \subseteq  D^{\leq 0},
\]
\item it is left $t$-exact for the Hodge module $t$-structures on the source and the perverse $t$-structure on the target:
\[
\GrDR(D^{\geq 0}) \subseteq  {}^pD^{\geq 0},
\]
\item $\GrDR(\bQ^H_X[d]) = \underline{\Omega}^0_X[d] \oplus \underline{\Omega}^1_X[d-1] \oplus \ldots \oplus \underline{\Omega}^d_X[0]$, where $\underline{\Omega}^i_X \in D^b_{\rm coh}(X)$ is the $i$-th Deligne Du Bois complex on $X$.
\end{enumerate}
\end{proposition}
\begin{proof}
See \cite[Lemma 7.3]{Schnell16} for why  $\GrDR$ produces a complex supported on $X$ and why it is well defined, i.e.\ independent of the embedding $u \colon X \to W$. Part (1) is clear as the Tate twist only shifts the filtrations.

 Part (2) follows from \cite[2.3.7]{Saito00} (see \cite[Theorem 28.1]{SchnellSurvey19} for the exact statement). Here, we point out that every projective map of varieties is $H$-projective (see \cite[Tag 01W8]{stacks-project}), and so it can be easily extended to a projective map between some smooth ambient spaces.
Part (3) follows from a formula of Saito (\cite[Section 2.4]{Saito88}, see also \cite[(9.2)]{popa2024injectivityvanishingdubois} for the exact statement):
\[
\bD_X \Gr_p\DR(\cM)  \cong \Gr_{-p}\DR(\bD_X \cM).
\]

Next (4) is clear. Indeed, by (\ref{eq:defGRDR}) we have that $\GrDR(\cM) \in D^{\leq 0}$ for every mixed Hodge module $\cM$ which immediately implies the claim. In turn, (5) follows from (4) and duality. Specifically, if $\cM$ is a mixed Hodge module, then
\[
\bD_X \GrDR(\cM) = \GrDR(\bD_X \cM ) \in D^{\leq 0}.
\]
Thus $\GrDR(\cM) \in {}^pD^{\geq 0}$ by (\ref{eq:perv-coh-dual}). 

Finally, (6) is a consequence of \cite[Theorem 4.2]{Saito00} (see \cite[p.\ 20]{popa2024injectivityvanishingdubois} for the exact statement).
\end{proof}

\begin{remark} \label{remark:grdr-graded}
The $\Gr_\kdot \DR$ functor extends to
\[
\Gr_\kdot \DR \colon D^b_{\rm MHM}(X) \to D^b_{\rm coh, {\rm gr}}(X)
\]
where $D^b_{\rm coh, {\rm gr}}(X)$ denotes the \emph{graded} bounded derived category of coherent sheaves on $X$. Parts (1), (2), (3), and (6) admit more specialised formulas:
\begin{align}
\Gr_\kdot \DR(\cM(k)) &= \Gr_{\kdot -k}\DR(\cM), \label{eq:GRDRTwist} \\
\Gr_{\kdot}\DR(f_*\cM) &\cong Rf_*\Gr_{\kdot}\DR(\cM), \\
\bD_X \Gr_\kdot\DR(\cM) &\cong \Gr_{-\kdot}\DR(\bD_X \cM), \text{ and }  \label{eq:GRDRDual} \\
\Gr_{-i} \DR(\bQ^H_X[d]) &\cong
\DDB^{i}_X[d-i]. \label{eq:GRDRTrivial}
\end{align}
In \eqref{eq:GRDRTrivial}, the indexing of the graded pieces may look unnatural to the reader at first, but this convention agrees with the calculation in (\ref{eq:defGRDR}). As it will be used many times in this article, let us point out that \eqref{eq:GRDRTrivial} immediately implies that:
\begin{equation} \label{eq:GRDRTrivial2}
\Gr_{-(d-i)} \DR(\bQ^H_X[d]) \cong
\DDB^{d-i}_X[i].
\end{equation}

\end{remark}


\begin{remark} \label{remark:jpush}
Let $X$ be a normal connected $d$-dimensional variety over $\bC$, let $f \colon Y \to X$ be a log resolution of singularities with exceptional divisor $E$, let $D$ be a reduced simple normal crossing divisor on $Y$ containing $E$, and let $j \colon U \hookrightarrow X$ be an inclusion of the smooth open subset $U \coloneqq X \, \backslash\, f(D)$. Let $j' \colon U \to Y$ be the natural inclusion of $U$ in $Y$. Then
\begin{equation} \label{eq:grdrjpush}
\Gr_{-i}\DR(j_*\bQ^H_U[d]) \cong Rf_*\Omega^i_Y(\log D)[d-i].
\end{equation}
Indeed,
\begin{align*}
\Gr_\kdot \DR(j_*\bQ^H_U[d]) &\cong  \Gr_\kdot\DR(f_*j'_*\bQ^H_U[d])\\
&\cong Rf_*\Gr_\kdot\DR(j'_*\bQ^H_U[d])   \\
&\cong \bigoplus^d_{i=0} Rf_*\Omega^i_Y(\log D)[d-i],
\end{align*}
where the first isomorphism is definitional, the second follows from Proposition \ref{prop:JakubsFavouriteFunctor}(2), and the last one is \cite[Theorem 1]{Saito07} (see also \cite[Theorem 6.1]{MP19}). By looking at the gradations it is easy to deduce the required isomorphism.
\end{remark}

It is natural to ask if we can sometimes get full $t$-exactness in Proposition \ref{prop:JakubsFavouriteFunctor}(4)(5). This is indeed the case for the lowest and the highest pieces of the Hodge filtration, namely we have the following result which we learned from \cite[Theorem 10.5]{popa2024injectivityvanishingdubois} and \cite{Park24}.

\begin{lemma} \label{lem:grdr-boundary}
Let $X$ be a reduced connected variety over $\bC$. Let $\cM \in D^b_{\rm MHM}(X)$ and let $k \in \bZ$. Then the following statements hold.
\begin{enumerate}
    \item If $\cM \in D^{\geq 0}_{\rm MHM}(X)$ and $\Gr_{i}\DR(\cM)=0$ for $i < k$, then
    \[
    \Gr_k\DR(\cM) \in D^{\geq 0}_{\rm coh}(X).
    \]
    \item If $\cM \in D^{\leq 0}_{\rm MHM}(X)$ and $\Gr_{i}\DR(\cM)=0$ for $i > k$, then
    \[
    \Gr_k\DR(\cM) \in {}^pD^{\leq 0}_{\rm coh}(X).
    \]
\end{enumerate}
\end{lemma}
\noindent In particular, if $\cM$ is a mixed Hodge module, then the lowest non-zero $\Gr_k\DR(\cM)$ is a coherent sheaf, and the highest non-zero $\Gr_k\DR(\cM)$ is a perverse coherent sheaf (see Proposition \ref{prop:JakubsFavouriteFunctor}(4)(5)). 
\begin{proof}
 We shall use the following formula that extends (\ref{eq:defGRDR}) to $D^b_{\rm MHM}(X)$ (see \cite[line below Lemma 9.6]{popa2024injectivityvanishingdubois}):
{\small \begin{equation} \label{eq:defGRDRFancy}
\Gr^F_k \DR(\cM) \coloneqq  \rm Tot\left(\Gr^F_k M^\kdot \to \Gr^F_{k+1} M^\kdot \otimes^L_{\cO_W} \Omega^1_W \to \cdots \to  \Gr^F_{k+{d_W}} M^\kdot \otimes^L_{\cO_W} \Omega_W^{d_W}\right).
\end{equation}}
\!\!Here $(M^\kdot, F) \in D^b_{\rm fil}(\cD_X)$ is the underlying object of the derived category of filtered D-modules. In (\ref{eq:defGRDRFancy}), the right hand side denotes the total complex of the double complex with vertical maps induced by the differential of a representative of $M^\kdot$ and  horizontal maps induced by the usual de Rham maps on each term of that representative. 

Let us start with (1). Since $\Gr_{<k}\DR(\cM)=0$, we get that:
\[
F_{<k + d_W} M^\kdot  = 0;
\]
see \cite[Lemma 9.6]{popa2024injectivityvanishingdubois} up to a shift from left to right $D$-modules 
or equivalently trace through  \eqref{eq:defGRDRFancy}. Therefore, by \eqref{eq:defGRDRFancy}: 
\[
\Gr_k\DR(\cM) = F_{k+d_W}M^\kdot  \otimes^L_{\cO_W} \Omega^{d_W}_W.
\]
Strictness of Hodge filtration yields $\cH^i(F_{k+d_W}M^\kdot) = F_{k+d_W}\cH^i(M^\kdot) =0$ for $i<0$, so
\[
\Gr_k\DR(\cM) \in D^{\geq 0}
\]
as required. 

Next, we infer (2) from (1) by duality. With the assumptions of (2) and using (\ref{eq:GRDRDual}), we get
\[
0 = \bD_X(\Gr_i\DR(\cM)) = \Gr_{-i}\DR(\bD_X(\cM))
\]
for all $i > k$. Since $\bD_X(\cM) \in D^{\geq 0}$ (see (\ref{eq:cons-self-dual}) and Subsection \ref{ss:HM}(1)(2)), we can apply (1) to get:
\[
\Gr_{-k}\DR(\bD_X(\cM)) \in D^{\geq 0}.
\]
But the left hand side is equal to $\bD_X(\Gr_k(\DR(\cM)))$ by (\ref{eq:GRDRDual}) again, and so
\[
\Gr_k\DR(\cM) \in {}^pD^{\leq 0}
\]
by (\ref{eq:perv-coh-dual}). This concludes the proof.
\end{proof}

\begin{remark} \label{remark:BL}
Fix an isomorphism between $\bC$ and some perfectoid field. The theory of $p$-adic Riemann-Hilbert correspondence (\cite{BhattLuriepadicRHmodp}, see \cite[Section 3.2]{BhattAbsoluteIntegralClosure} and \cite[Theorem 4.2]{BMPSTWW2}) provides us with a functor:
\[
{\rm RH}^{\rm Higgs} \colon D^b_{\rm cons}(X,\bQ_p) \to D^b_{\rm coh, {\rm gr}}(X)
\]
that satisfies all the properties of the $\Gr_{\kdot}\DR$ functor from Proposition \ref{prop:JakubsFavouriteFunctor} and Remark \ref{remark:grdr-graded}. Since the $\Gr_\kdot \DR$ functor is the main tool in this article that requires Hodge module theory, we could have replaced most of the uses of Hodge modules with $\bQ_p$-constructible sheaves. This reinforces the point made earlier, that from the viewpoint of our article, Hodge modules may be thought of as perverse $\bQ$-constructible sheaves with \emph{some additional data}.

Informally speaking, the reason that the ${\rm RH}^{\rm Higgs}$ functor has only constructible sheaves as its input is that the data of additional filtrations inherent to Hodge modules is implicitly replaced by the action of the Galois group on $\bQ_p$-sheaves.
\end{remark}

\subsection{Intersection and dual Deligne--Du Bois complexes}
Recall that the duality functor on $D^b_{\rm coh}(X)$ is defined as follows $\bD_X(-) = R\Hom_{\cO_X}(-, \omega^\kdot_X)$.
Following \cite{PP24}, we introduce the following definition.
\begin{definition}[{\cite[Definition 3.3]{PP24}}] \label{def:io}
Let $X$ be a normal connected variety of dimension $d$ defined over $\bC$. We define the \emph{$i$-th intersection Du Bois complex} $\IO^i_X$ so that
\[
\Gr \DR(\IC^H_X) = \IO^0_X[d] \oplus \IO^1_X[d-1] \oplus \cdots \oplus \IO^d_X[0].
\]
Specifically, $\IO^i_X \coloneqq \Gr_{-i}\DR(\IC^H_X)[-d+i]$. 
\end{definition}

Let us recall the key properties of $\IO^i_X$ from \cite{KS21} (see also \cite{Park24}).
\begin{proposition}[{\cite[Section 8]{KS21}, \cite[Lemma 3.4]{PP24}}] \label{prop:all-properties-of-IO}
Let $X$ be a normal connected variety of dimension $d$ defined over $\bC$. Let $f \colon Y \to X$ be a resolution of singularities. Then  for every $0 \leq i \leq d$:
\begin{enumerate}
    \item $\IO^0_X = Rf_*\cO_Y$,
    \item $\IO^d_X = f_*\omega_Y$,
    \item $\cH^0(\IO^i_X) = f_*\Omega^i_Y$,
    \item $\IO^i_X$ is a direct summand of $Rf_*\Omega^i_Y$, and
    \item $\bD_X(\IO^i_X[d-i]) = \IO^{d-i}_X[i]$.
\end{enumerate}
\end{proposition}
\noindent For the convenience of the reader, we sketch the argument below.
\begin{proof}
We start with (5). By applying $\Gr_\kdot \DR$ to (\ref{eq:dual-IC}), we get
\[
\Gr_{i}\DR(\bD_X(\IC^H_X)) \cong \Gr_{i}\DR(\IC^H_X(d)).
\]
The right hand side is equal to $\Gr_{-d+i}\DR(\IC^H_X) = \IO^{d-i}_X[i]$. The left hand side equals
\[
\bD_X(\Gr_{-i}\DR(\IC^H_X)) = \bD_X(\IO^i_X[d-i])
\]
by \eqref{eq:GRDRDual}. This concludes the proof of (5).

Next, by Saito's decomposition theorem \cite[Theorem 15.1 and Theorem 16.1] {Schnell16} (generalising \cite{BBDG18}), we have that:
\[
f_*\bQ^H_Y[d] = \IC^H_X \oplus N,
\]
where $N$ is a Hodge module supported on a strictly closed subvariety $Z \subsetneq X$. Thus
\begin{align*}
\Gr_\kdot\DR(\IC^H_X) \oplus \Gr_\kdot\DR(N) &= \Gr_\kdot\DR(f_*\bQ^H_Y[d]) \\
&= Rf_*\Gr_\kdot\DR(\bQ^H_Y[d]) \\
&= \bigoplus^d_{i=0} Rf_*\Omega^i_Y[d-i],
\end{align*}
by Proposition \ref{prop:JakubsFavouriteFunctor}(2)(6) (see also Remark \ref{remark:grdr-graded}). Therefore
\[
Rf_*\Omega^i_Y = \IO^i_X \oplus \Gr_{-i}\DR(N)[i-d],
\]
which shows (4). Moreover, since $f_*\Omega^i_Y$ is torsion-free and $\Gr_{-i}\DR(N)$ is supported on $Z$, applying $\cH^0$ yields:
\[
f_*\Omega^i_Y = \cH^0(\IO^i_X)
\]
concluding (3).

By Grauert-Riemenschneider vanishing $Rf_*\omega_Y = f_*\omega_Y$, and so (4) implies that $\IO^d_X$ is a sheaf supported in degree $0$. Hence by (3), $\IO^d_X = \pi_*\omega_Y$. Finally (1) follows automatically from (2) and (5):
\[
\IO^0_X[d] \cong \bD_X(\IO^d_X) \cong \bD_X(Rf_*\omega_Y) \cong Rf_*\cO_Y[d]. \qedhere
\]
\end{proof}

In addition to $\DDB^i_X$ and $\IO^i_X$ we will also study the dual Deligne--Du Bois complex.
\begin{definition} \label{def:do}
{For a normal connected variety $X$,} define
\[
\DO^i_X \coloneqq \bD_X(\DDB^{d-i}_X[d]) \in D^b_{\rm coh}(X).
\]
\end{definition}
\noindent When $X$ is smooth, $\DO^i_X \cong \Omega^i_X$.\\

Note that $\DO^i_X[d-i] = \bD_X(\DDB^{d-i}_X[i])$. Hence, since $\GrDR$ commutes with duality (Proposition \ref{prop:JakubsFavouriteFunctor}(3)), we get:
\begin{align*}
\GrDR(\bD_X(\bQ^H_X[d])) &\cong \bD_X(\DDB^0_X[d] \oplus \DDB^1_X[d-1] \oplus \cdots \oplus \DDB^d_X[0]) \\
&\cong \DO^0_X[d] \oplus \DO^1_X[d-1] \oplus \cdots \oplus \DO^d_X[0].
\end{align*}

Whenever $i < {\rm codim}_X({\rm Sing}(X))$, we can calculate $\DO^i_X$ explicitly. To this end, we first establish the following lemma which includes the study of non-normal varieties. Recall that the dimension of a variety $Z$ is the maximum of the dimensions of $\cO_{Z,z}$ for every closed point $z \in Z$.
\begin{lemma} \label{lemma:DOnon-normal}
    Let $Z$ be a reduced variety of dimension $d$ defined over $\bC$. Then 
    \[
    \bD_Z(\underline{\Omega}^{i}_Z)\in D^{\geq -d}_{\rm coh}(Z)
    \]
    for every integer $i \geq 0$.
\end{lemma}
\begin{proof}
    We consider the following statements:
    \begin{enumerate}
        \item[$(N_d)$:] $\bD_X(\underline{\Omega}^{i}_X)\in D^{\geq -\dim X}_{\rm coh}(X)$ for all normal connected varieties $X$ with $\dim X\leq d$.
        \item[$(V_d)$:] $\bD_Z(\underline{\Omega}^{i}_Z)\in D^{\geq -\dim Z}_{\rm coh}(Z)$ for all {reduced} varieties $Z$ with $\dim Z\leq d$.
    \end{enumerate}
    The lemma follows from the following claim.
    \begin{claim*}
        Assume that the statements $(N_{d-1})$ and $(V_{d-2})$ hold. Then $(N_{d})$ and $(V_{d-1})$ hold as well.
    \end{claim*}
         We first prove that $(V_{d-1})$ holds. Let $Z$ be a {reduced} variety of dimension at most $d-1$.
        Now, since $\DDB^i_Z$ is a (derived) $h$-sheaf, we can consider the following blow-up square and the induced homotopy pullback square:
\[
\begin{tikzcd}
\cD \ar{r} \ar{d} & Z^{\nu} \ar{d}{\nu} \\
\cC \ar{r} & Z.
\end{tikzcd}
\hspace{4em}
\begin{tikzcd}
\nu_*\DDB^{i}_{\cD}   & \ar{l} \nu_*\DDB^{i}_{Z^{\nu}}  \\
\DDB^{i}_{\cC}  \ar{u} & \ar{l}  \ar{u}[swap]{\nu^*} \DDB^{i}_Z.
\end{tikzcd}
\]
where $\nu \colon Z^{\nu} \to Z$ is the normalisation and $\cC$, $\cD$ are conductors (cf. \cite[Section 2.6]{W24Rel}). Note that $\dim \cC$ and $\dim \cD$ are at most $d-2$, and $\dim Z^{\nu}=d-1$. We have an exact triangle
\[
\DDB^{i}_Z \to \DDB^{i}_C \oplus \nu_*\DDB^{i}_{Z^{\nu}} \to \nu_*\DDB^{i}_D \xrightarrow{+1}
\]
and its dual:
\begin{equation*} 
\bD_Z(\nu_*\DDB^{i}_D) \to \bD_Z(\DDB^{i}_C) \oplus \bD_Z(\nu_*\DDB^{i}_{Z^{\nu}}) \to \bD_Z(\DDB^{i}_Z) \xrightarrow{+1}
\end{equation*} 
By Grothendieck duality, this triangle is equivalent to:
\begin{equation} \label{eq:dualetr1}
\bD_D(\DDB^{i}_D) \to \bD_C(\DDB^{i}_C) \oplus \bD_{Z^{\nu}}(\DDB^{i}_{Z^{\nu}}) \to \bD_Z(\DDB^{i}_Z) \xrightarrow{+1}
\end{equation} 
Then 
\begin{alignat*}{3}
&\bD_{Z^\nu}(\DDB^{i}_{Z^{\nu}})  \in D^{\geq -\dim Z} &&\text{by Assumption $(N_{d-1})$, and}\\
&\bD_{\cC}(\DDB^i_{\cC})\in D^{\geq -\dim \cC}\, \text{ and }\, \bD_{\cD}(\DDB^{i}_{\cD}) \in D^{\geq -\dim \cD} \quad &&\text{by Assumption $(V_{d-2})$}.
\end{alignat*}
Thus, we get $\bD_Z(\DDB^{i}_Z) \in D^{\geq -\dim Z}$  by \eqref{eq:dualetr1}, and so $(V_{d-1})$ holds.

Next, we prove that $(N_d)$ holds. Let $X$ be a normal connected variety of $\dim X\leq d$.
Let $f\colon Y\to X$ be a log resolution and $Z\coloneqq f(E)$. Then we have the following blow-up square and the induced homotopy pullback square:
\[
\begin{tikzcd}
E \ar{r} \ar{d} & Y \ar{d}{f} \\
Z \ar{r} & X.
\end{tikzcd}
\hspace{4em}
\begin{tikzcd}
Rf_*\DDB^{i}_{E}   & \ar{l} Rf_*\Omega^{i}_{Y}  \\
\DDB^{i}_{Z}  \ar{u} & \ar{l}  \ar{u}[swap]{f^*} \DDB^{i}_X.
\end{tikzcd}
\]
Since homotopy fibres of horizontal arrows in a homotopy pullback square are isomorphic\footnote{the reader can also use the octahedral axiom here instead}, we get the exact triangle:
\[
Rf_{*}\Omega^{i}_Y(\log E)(-E) \to \underline{\Omega}^{i}_X \to \underline{\Omega}^{i}_Z \xrightarrow{+1}.
\]
By taking $\bD_X(-)$, we get the exact triangle
\begin{equation} \label{eq:triangledualnonnormal}
\bD_Z(\underline{\Omega}^{i}_Z)\to \bD_X(\underline{\Omega}^{i}_X)\to Rf_{*}\Omega_Y^{d-i}(\log E)[\dim X]\xrightarrow{+1}.
\end{equation}
Hence, by $(V_{d-1})$ and the above exact triangle, we conclude that $\bD_X(\underline{\Omega}^{i}_X)\in D^{\geq -\dim X}$. Thus, $(N_d)$ holds.
\end{proof}

\begin{proposition}[cf.\ {\cite[Lemma 2.4]{SVV}}] \label{prop:dualDDB-explicit}
Let $X$ be normal connected variety of dimension $d$ defined over $\bC$. Let $f \colon Y \to X$ be a log resolution of singularities with exceptional locus $E$ and $Z \coloneqq f(E)$. Then
\begin{enumerate}
    \item $\DO^i_X \in D^{\geq 0}$,
    \item there exists a natural map:
    \[
    \theta \colon \DO^i_X \to Rf_*\Omega^i_Y(\log E)
    \]
    \item $\cH^0(\DO^i_X)=f_*\Omega^i_Y(\log E)$, and so it is torsion-free.
\end{enumerate} 
Moreover, the map $\theta$ is an isomorphism if $i < {\rm codim}_X Z$.
\end{proposition}
\noindent In particular, $\DO^i_X \cong Rf_*\Omega^i_Y(\log E)$ if $f$ is a strong resolution of singularities and $i < {\rm codim}_X({\rm Sing}(X))$.

\begin{remark}
The above map $\theta$ was already constructed in (\ref{eq:triangledualnonnormal}) as part of the proof of Lemma \ref{lemma:DOnon-normal} by using that $\DDB^i_X$ is a (derived) $h$-sheaf (see also \cite[Lemma 2.4]{SVV}).  Below, we recreate its construction using the theory of Hodge modules.  \emph{A priori}, it is not clear if these two maps agree. We shall not do this verification in our paper. Unless otherwise stated, all the maps we use in this article are constructed by way of Hodge modules and \emph{not} using the $h$-sheaf property of $\DDB^i_X$.
\end{remark}
\begin{proof}[{Proof of Proposition \ref{prop:dualDDB-explicit}}]
First, (1) follows from Lemma \ref{lemma:DOnon-normal}. Next, we move to the proof of (2). 
To this end, consider the following exact triangle:
\[
j_!\bQ^H_U[d] \to \bQ^H_X[d] \to i_*\bQ^H_Z[d]  \xrightarrow{+1}
\]
for $j \colon U \coloneqq X \setminus Z \hookrightarrow X$ and $i \colon Z \hookrightarrow X$ being natural inclusions (see Subsection \ref{ss:HM}(3) and Remark \ref{remark:exact-triangles}). By applying Hodge module duality $\bD_X(-)$, we get an exact triangle:
\[
i_*\bD_Z(\bQ^H_Z[d]) \to \bD_X(\bQ^H_X[d]) \to j_*\bQ^H_U[d](d) \xrightarrow{+1}
\]
where 
\[
\bD_X(j_!\bQ^H_U[d]) = j_*\bD_X(\bQ^H_U[d]) = j_*\bQ^H_U[d](d)
\]
by Subsection \ref{ss:HM}(3)(4). Thus, by applying $\Gr_{d-i}\DR$ and using (\ref{eq:GRDRDual}), (\ref{eq:GRDRTrivial}), and (\ref{eq:grdrjpush}) we obtain:
\[
\bD_Z(\DDB^{d-i}_Z[i]) \to \bD_X(\DDB^{d-i}_X[i]) \xrightarrow{\theta} Rf_*\Omega^i_Y(\log E)[d-i] \xrightarrow{+1},
\]
where for the first term we used:
\begin{align*}
\Gr_{d-i}\DR(\bD_Z(\bQ^H_Z[d])) &= \bD_Z(\Gr_{-(d-i)}\DR(\bQ^H_Z[d])) \\
&= \bD_Z(\Gr_{-(d-i)}\DR(\bQ^H_Z[\dim Z])[d-\dim Z])\\
&= \bD_Z(\DDB^{d-i}_Z[\dim Z - (d-i) + d-\dim Z])\\
&= \bD_Z(\DDB^{d-i}_Z[i]).    
\end{align*}
This yields an exact triangle:
\begin{equation} \label{eq:triangledualnonnormalHM}
\bD_Z(\DDB^{d-i}_Z[d]) \to \DO^{i}_X \xrightarrow{\theta} Rf_*\Omega^i_Y(\log E) \xrightarrow{+1},
\end{equation}
and concludes the construction of $\theta$. Moreover, $\theta$ is an isomorphism when $d-i > \dim Z$, or equivalently $i < {\rm codim}_X(Z)$.\\

We are left to show (3). 
By  using Lemma \ref{lemma:DOnon-normal}, we see that 
 \[
 \bD_{Z}(\DDB^{d-i}_Z[d]) \in D^{\geq 2}
 \]
 given that $\dim Z \leq d-2$ by normality of $X$. In particular, by the exact triangle (\ref{eq:triangledualnonnormalHM}), we have that  $\DO^i_X \in D^{\geq 0}$ and by applying $\cH^0$, we get that   $\cH^0(\DO^i_X) = f_*\Omega^i_Y(\log E)$.
\end{proof}

\begin{remark} \label{rem:fact}
We have a natural factorisation:
\[
\bQ^H_X[d] \xrightarrow{\phi} \IC^H_X \xrightarrow{\cong} \bD_X(\IC^H_X)(-d) \xrightarrow{\bD_X(\phi)} \bD_X(\bQ^H_X[d])(-d).
\]
By applying $\Gr_\kdot\DR$ we thus get a factorisation
\[
\DDB^i_X \xrightarrow{\psi} \IO^i_X \xrightarrow{\cong} \bD_X(\IO^{d-i}_X[d]) \xrightarrow{\bD_X(\psi)} \bD_X(\DDB^{d-i}_X[d]) =  \DO^i_X.
\]
Combined with natural maps $\Omega^i_X \to \Omega^i_{X,h} \to \DDB^i_X$ and the map from Proposition \ref{prop:dualDDB-explicit}, we get the factorisation:
\begin{equation} \label{eq:fact}
\Omega^i_X \to \Omega^i_{X,h} \to \DDB^i_X \to \IO^i_X \to \DO^i_X \to Rf_*\Omega^i_Y(\log E).
\end{equation}
Let us emphasise that the composite map $\Omega^i_X \to Rf_*\Omega^i_Y(\log E)$ agrees with the usual pullback map $f^*$. Indeed, because $\Omega^i_X$ is concentrated in degree $0$, we need to verify that the induced map $\Omega^i_X \to \cH^0(Rf_*\Omega^i_Y(\log E)) = f_*\Omega^i_Y(\log E)$ agrees with the pullback map $f^*$, but because the latter term is torsion-free, this can be checked on the smooth locus, for which this is clear.
\end{remark}



Finally, let us restate Kebekus--Schnell's extension theorem.
\begin{proposition} \label{prop:KebekusSchnell}
Let $X$ be a normal connected variety of dimension $d$ defined over $\bC$. Suppose that $X$ has rational singularities and let $f \colon Y \to X$ be a log resolution of singularities. Then the natural maps of sheaves induced from (\ref{eq:fact}):
\[
\Omega^i_{X,h} = \cH^0(\DDB^i_{X}) \to \cH^0(\IO^i_X) \to \cH^0(\DO^i_X) \to f_*\Omega^i_Y(\log E)
\]
are all equalities for every $i \geq 0$, with all the sheaves equal to $\Omega^{[i]}_X$.
\end{proposition}
\begin{proof}    
By \cite[Corollary 1.11]{KS21},
 we get that $\Omega^i_{X,h} = \Omega^{[i]}_X$.
Hence the statement of the proposition immediately follows given that $f_*\Omega^i_Y(\log E)$, $\cH^0(\IO^i_X)$, and $\cH^0(\DO^i_X)$ are torsion-free (see Proposition \ref{prop:all-properties-of-IO}(3) and Proposition \ref{prop:dualDDB-explicit}).
\end{proof}


    %

\subsection{Vanishing results}

In this subsection, we list some fundamental properties of $\DDB^i_X$, $\IO^i_X$, and $\DO^i_X$. We emphasise that most of the proofs follow immediately from left or right $t$-exactness of $\GrDR$.
\begin{lemma} \label{lem:DDB-IO-coh-amplitude}
Let $X$ be a normal connected variety of dimension $d$ defined over $\bC$. Then both $\DDB^i_X$ and $\IO^i_X$ are of cohomological amplitude $[0,d-i]$, specifically, they belong to:
\[
D^{[0,d-i]}_{\rm coh}(X).
\]
\end{lemma}
\noindent That $\DDB^i_X \in D^{[0,d-i]}_{\rm coh}(X)$ is called Steenbrink's vanishing (\cite[Theorem 7.29]{Peter-Steenbrink(Book)}).
\begin{proof}
Recall that $\bQ_X[d]$ and $\IC_X$ belong to ${}^pD^{\leq 0}_{\rm cons}(X,\bQ)$. This follows by (\ref{eq:trivial-q-leq0}) in the former case and the fact that $\IC_X$ is perverse in the latter. 
In particular, $\bQ^H_X[d]$ and $\IC^H_X$ belong to $D^{\leq 0}_{\rm MHM}(X)$ (see Subsection \ref{ss:HM}(1)(2)). Therefore, both 
\[
\GrDR(\IC^H_X),\, \GrDR(\bQ^H_X[d]) \in D^{\leq 0}_{\rm coh}(X)
\]
by the right $t$-exactness of $\GrDR$ (Proposition \ref{prop:JakubsFavouriteFunctor}(4)). Hence
\[
\DDB^i_X,\, \IO^i_X \in D^{\leq d-i}_{\rm coh}(X)
\]
by Proposition \ref{prop:JakubsFavouriteFunctor}(6) and Definition \ref{def:io}. The fact that 
\[
\DDB^i_X,\, \IO^i_X \in D^{\geq 0}_{\rm coh}(X)
\]
is clear by construction and Proposition \ref{prop:all-properties-of-IO}(4), respectively.
\end{proof}

    

\begin{lemma} \label{lem:loccohdual-and-int} 
Let $X$ be a normal connected variety of dimension $d$ defined over $\bC$. Then both $\IO^i_X$ and $\DO^{i}_X$ are of perverse cohomological amplitude $\geq d-i$. Specifically, this means that
\[
H^j_\m(\IO^i_X) = 0 \quad \text{ and } \quad H^j_\m(\DO^{i}_X) = 0
\]
for every maximal ideal $\m$ on $X$ and all integers $i, j \geq 0$ such that $i+j < d$.
\end{lemma}
\begin{proof}
Note that both $\bD_X(\bQ_X[d])$ and $\IC_X$ belong to ${}^pD^{\geq 0}_{\rm cons}(X,\bQ)$. Indeed, the former case follows from (\ref{eq:trivial-q-leq0}):
\[
\bQ_X[d] \in {}^pD^{\leq 0}_{\rm cons}(X,\bQ)
\]
and the fact that
the $t$-structure on $D^{b}_{\rm cons}(X,\bQ)$ is self-dual \eqref{eq:cons-self-dual}. The latter case follows immediately from the fact that $\IC_X$ is perverse.

In particular, $\bD_X(\bQ^H_X[d])$ and $\IC^H_X$ belong to $D^{\geq 0}_{\rm MHM}(X)$. Therefore, both 
\[
\GrDR(\bD_X(\bQ^H_X[d])), \, \GrDR(\IC^H_X) \in {}^pD^{\geq 0}_{\rm coh}(X)
\]
by the left $t$-exactness of $\GrDR$ (Proposition \ref{prop:JakubsFavouriteFunctor}(5)). Hence
\[
\DO^i_X,\, \IO^i_X \in {}^pD^{\geq d-i}_{\rm coh}(X)
\]
by Definition \ref{def:do} and Definition \ref{def:io}. Now the statement of the lemma follows from the definition of the perverse coherent $t$-structure (Definition \ref{def:perv-cons-t} or Remark  \ref{rem:perv-coh-maximal-only}).
\end{proof}

\begin{remark}
Lemma \ref{lem:loccohdual-and-int} immediately implies Flenner's extension theorem \cite{Flenner88}: if $X$ is a normal connected variety of dimension $d$ satisfying $\codim_X(\mathrm{Sing}(X))\geq i+2$, then $f_{*}\Omega^i_Y\cong \Omega^{[i]}_X$.
    To see this, it suffices to show that $f_{*}\Omega^i_Y$ is $S_2$, that is, $H^1_{\p}\left((\pi_{*}\Omega^i_Y)_{\p}\right)=0$ for all $\p\in\mathrm{Sing}(X)$.
    Following a technique of Kebekus--Schnell \cite[Proposition 6.4]{KS21}, the vanishing is reduced to
    \[
    H^1_{\p}(\IO^i_{X,\p})=0
    \]
    for all $\p\in\mathrm{Sing}(X)$.
    Indeed, by considering the exact triangle
    \[
\pi_{*}\Omega^i_Y(=\cH^0(\IO^i_X))\to \IO^i_X\to C\xrightarrow{+1},
    \]
    we can observe that $C\in D^{>0}_{\mathrm{coh}}(X)$ and thus $H^0_{\p}(C_{\p})=0$.
    
    Now, by Lemma \ref{lem:loccohdual-and-int}, we have
    \[
    \IO^i_X\in {}^pD^{\geq d-i}_{\mathrm{coh}}(X),
    \]
    which shows that 
    \[
    H^j_{\p}(\IO^i_{X,\p})=0
    \]
    for all $j<d-i-\dim V(\p)$.
    Thus, we get 
    \[
    H^1_{\p}(\IO^i_{X,\p})=0
    \]
    for all $i\geq 0$ since $\dim V(\p)\leq d-i-2$ by assumption.
\end{remark}

\begin{proposition} \label{prop:vanishing} Let $X$ be normal connected variety of dimension $d$ defined over $\bC$ and let $f \colon Y \to X$ be a log resolution of singularities with exceptional divisor $E$. Then $R\pi_*\Omega^i_Y(\log E)$ is of perverse cohomological amplitude $\geq d - i$. In other words:
\[
H^j_\m(R\pi_*\Omega^i_Y(\log E)) = 0
\]
for every maximal ideal $\m$ and all $i, j \geq 0$ such that $i+j < d$.
\end{proposition}
One can obtain Steenbrink's vanishing by the proposition and local duality and as in \cite[Theorem 3.2]{kovacs13}. Thus, the proposition give another proof of Steenbrink's vanishing using the properties of $\GrDR$.
\begin{proof}


Let $j' \colon U \to Y$ be the inclusion of the complement $U$ of the divisor $E$ and let $j \colon U \to X$ be the inclusion into $X$. Note that $f \circ j' = j$. Since $j_*$ is left $t$-exact (see \cite[1.4.16 (i)]{BBDG18}, \cite[4.2.4]{BBDG18}, or \cite[Remark 3.11(b)]{BMPSTWW2}), we get that 
\[
j_*\bQ_U[d] \in {}^pD^{\geq 0}_{\rm cons}(X,\bQ),
\]
and thus $j_*\bQ^H_U[d]\in D^{\geq 0}_{\rm MHM}(X)$ by the $t$-exactness and faithfulness of ${\rm rat}$.
Therefore,
\begin{align*}
\bigoplus^d_{i=0} Rf_*\Omega^i_Y(\log E)[d-i] 
\cong \GrDR(j_*\bQ^H_U[d]) \in {}^pD^{\geq 0}_{\rm coh}(X),
\end{align*}
where the first isomorphism follows from Remark \ref{remark:jpush}
and the last inclusion follows from the left $t$-exactness of $\GrDR$ (Proposition \ref{prop:JakubsFavouriteFunctor}(5)).

This translates to the exact vanishing on local cohomology as required.
\end{proof}

\subsection{Restriction to general hyperplane}
First, we recall a result describing a restriction of a Hodge module to a general hyperplane. 
\begin{proposition} \label{prop:restriction-hodge-module}
Let $X$ be a normal quasi-projective variety over $\bC$ and let $\cM$ be a mixed Hodge module. Let $H$ be a general member of a base-point-free linear system, let $i \colon H \hookrightarrow X$ be the natural inclusion, and let $\cM_H \coloneqq \cH^{-1}(i^*\cM)$. Then there exists an exact triangle:
\[
\Gr_{\kdot+1}\DR(\cM_H) \to \GrDR(\cM)\otimes^L_{\cO_X}\! \cO_H \to \Gr_{\kdot}\DR(\cM_H)[1] \xrightarrow{+1}.
\]
\end{proposition}

\begin{proof}
Since $H$ is a general member of a base-point-free linear system, it is non-characteristic for $\cM$ (see \cite[Definition 4.14]{KS21}).
Thus the proposition follows from \cite[Proposition 4.17]{KS21}. Here, we are implicitly using the fact that $H$ comes as a restriction of some divisor under an embedding $u \colon X \to W$ to a smooth ambient space $W$ (see Remark \ref{rem:linsys-rest}).
\end{proof}
We warn the reader that we have not verified that the above exact triangle is independent of the choice of an embedding $u \colon X \to W$.

\begin{corollary}[{\cite[Lemma 3.2]{SVV} and \cite[Lemma 6.6]{Park24}}] \label{cor:restrictionDDB-and-IO}
Let $X$ be a normal quasi-projective variety over $\bC$ and let $H$ be a general member of a base-point-free linear system. Then for every integer $i \geq 0$, there exist the following exact triangles
\begin{align*}
\DDB^{i-1}_H \otimes^L_{\cO_H} \cO_H(-H) \to \DDB^{i}_X&\otimes^L_{\cO_X}\!\cO_H \to \DDB^i_H \xrightarrow{+1} \\
\IO^{i-1}_H \otimes^L_{\cO_H} \cO_H(-H) \to \IO^{i}_X& \otimes^L_{\cO_X}\!\cO_H \to \IO^i_H \xrightarrow{+1}.
\end{align*}
\end{corollary}
\begin{proof}
Set $d = \dim X$. The former triangle exists by the same proof as in \cite[Lemma 3.2]{SVV}, and the latter by  \cite[Lemma 6.6]{Park24}.
\end{proof}

\begin{remark}
The existence of the latter exact triangle follows immediately from Proposition \ref{prop:restriction-hodge-module} given that $\cH^{-1}(i^*\IC^H_X) = \IC^H_H$ by Saito's theory (see \cite[Proof of Lemma 6.6]{Park24}). Similarly, one can obtain the former exact triangle by using a derived category variant of Proposition \ref{prop:restriction-hodge-module}. However, since there is no readily available reference for this, we do not pursue such an approach here. Instead, we refer to \cite[Lemma 3.2]{SVV}, which establishes such an exact triangle using the $h$-sheaf property of $\DDB^i_X$. For this reason, it is \emph{a priori} not clear whether these two exact triangles are compatible in any way. However, this will not be needed in our article --- we will use only the existence of the exact triangles and will not concern ourselves with the maps themselves ---and so we do not verify this here.
\end{remark}

\subsection{Miscellaneous}
We collect a few easy facts about local cohomology and restriction to a general hyperplane.

\begin{lemma} \label{lem:general-res-coh}
Let $R$ be a ring of finite type over a field which satisfies $\dim R>0$ and let $M \in D_{\rm fg}({\rm Mod}(R))$ be a bounded complex of finitely generated $R$-modules. Fix a general element $t \in R$;
specifically, we require that for all integers $k$:
\begin{equation} \label{eq:assump-general-res-coh}
 {\text{ both } R \ \text{and}\ }\cH^k(M) \text{ are $t$-torsion free.}
\end{equation}
Then for all integers $k$:
\[
\cH^k(M \otimes^L_R R/t) \cong \cH^k(M) \otimes_R R/t.
\]
\end{lemma}
\noindent Note that $\cH^k(M)$ being $t$-torsion free is the same as saying that none of the associated primes of $\cH^k(M)$ contain $t$. Here there are only finitely many associated primes of $\cH^k(M)$ for a fixed $k$ as this module is finite over $R$. 
\begin{proof}

We denote the $i$-th term of a complex representing $M$ by  $M_i$. Since every module admits a resolution by free modules (possibly of infinite length) and since $M$ is bounded, using the standard total complex method we may replace $M$ by a quasi-isomorphic complex bounded from above such that $M_i$ are free for all $i$.

We have the following sequence:
\[
\cdots \to M_{k-1} \to M_k \to M_{k+1} \to \cdots
\]
Now make the following definition:
\[
B_k \coloneqq {\rm im}(M_{k-1} \to M_k) \quad \text{ and } \quad Z_k \coloneqq {\rm ker}(M_k \to M_{k+1}).
\]
We get a short exact sequence:
\[
0 \to B_k \to Z_k \to \cH^k(M) \to 0.
\]
\begin{claim} The following identities hold.
\begin{align}   
B_k \otimes_R R/t &= {\rm im}(M_{k-1} \otimes_R R/t \to M_{k} \otimes_R R/t) \label{eq:bktensor}\\
Z_k \otimes_R R/t &= {\rm ker}(M_{k} \otimes_R R/t \to M_{k+1} \otimes_R R/t) \label{eq:zktensor}.
\end{align}
\end{claim}
\begin{proof}
The prove the first equality (\ref{eq:bktensor}), consider the following maps:
\[
M_{k-1} \twoheadrightarrow B_k \hookrightarrow M_k, 
\]
which after tensoring by $R/t$ yield
\[
M_{k-1} \otimes_R R/t \twoheadrightarrow B_k \otimes_R R/t \to M_k \otimes_R R/t.
\]
To conclude the proof of the first equality \eqref{eq:bktensor}, we need to check that $B_k \otimes_R R/t \to M_k \otimes_R R/t$ is injective. To this end, consider the following short exact sequences:
\begin{align*}
&0 \to B_k \to M_k \to M_k/B_k \to 0, \text{ and }\\
&0 \to \cH^k(M) \cong Z_k/B_k \to M_k/B_k \to B_{k+1} \to 0.
\end{align*}
Since $B_{k+1} \subseteq M_{k+1}$ and $\cH^k(M)$ are $t$-torsion-free, so is  $M_k/B_k$. In particular, by looking at the long-exact sequence of Tor, we get that $B_k \otimes_R R/t \to M_k \otimes_R R/t$ is injective  as required to conclude the proof of the first equality \eqref{eq:bktensor}.

As for the second equality (\ref{eq:zktensor}), consider the following short exact sequence:
\[
0 \to Z_k \to M_k \to B_{k+1} \to 0.
\]
Since $B_{k+1} \subseteq M_{k+1}$ is $t$-torsion-free, the long exact sequence for Tor as above,  yields the short exact sequence
\[
0 \to Z_k \otimes_R R/t \to M_k \otimes_R R/t \to B_{k+1} \otimes_R R/t \to 0
\]
with $B_{k+1} \otimes_R R/t \subseteq M_{k+1} \otimes_R R/t$ as proven above. This immediately concludes the proof of the second equality (\ref{eq:zktensor}).   
\end{proof}

By assumption (\ref{eq:assump-general-res-coh}) and long exact sequence for Tor we get that:
\begin{equation} \label{eq:ses-general-res-coh}
0 \to B_k \otimes_R R/t \to Z_k \otimes_R R/t \to \cH^k(M) \otimes_R R/t \to 0
\end{equation}
is exact. Since $M_i$ were chosen to be free, we may represent $M \otimes^L_R R/t$ by a complex $N$ whose $i$-term $N^i$ is equal to $M^i \otimes_R R/t$ for all $i$. 

Therefore, by the above claim $B_k \otimes_R R/t$ and $Z_k \otimes_R R/t$ are the boundaries and cycles at the $k$-term of the complex $N$, respectively, In particular, by (\ref{eq:ses-general-res-coh}):
\[
\cH^k(M \otimes^L_R R/t) \cong \cH^k(M) \otimes_R R/t. \qedhere
\]
\end{proof}

\begin{lemma} \label{lem:lc-max-to-prime}
Let $(R,\m)$ be a local excellent ring and let $M \in D^b_{\rm fg}({\rm Mod}(R))$ be a bounded complex of finitely generated $R$-modules. Fix an integer $k>0$ and suppose that $H^i_\m(M) = 0$ for all integers $0 \leq i \leq k$. Then
\[
H^i_{\p}(M) = 0
\]
for all prime ideals $\p$ and all integers $0 \leq i \leq k - \dim R/\p$.
\end{lemma}
\begin{proof}
By Remark \ref{rem:perv-coh-maximal-only}, we have that $M \in {}^pD^{> k}_{\rm fg}({\rm Mod}(R))$. Hence, by Definition \ref{def:perv-cons-t}:
\[
R\Gamma_{\p}(M) \in D^{> k - \dim R/p}_{\rm fg}({\rm Mod}(R))
\]
which concludes the proof.
\end{proof}

\begin{lemma} \label{lem:lc-restrict-to-divisor}
Let $R$ be a ring of finite type over a field with $\dim R>0$, let $M, N \in D^b_{\rm fg}({\rm Mod}(R))$ be two bounded complex of finitely generated $R$-modules, and  let $\phi \colon M \to N$ be a map between them. Fix an integer $k \geq 0$ and a general element $t \in R$; specifically, we require that both
\begin{align*}
&{\rm Ext}^{-k+1}_{R}(M, \omega^\kdot_{R}) \quad \text{ and } \quad {\rm Ext}^{-k+1}_{R}(N, \omega^\kdot_{R})
\end{align*}
are $t$-torsion free. Fix a maximal ideal $\m$ of $R$.
Now, suppose that the map
\[
\phi \colon H^k_{\m}(M) \to H^k_{\m}(N)
\]
is injective. Then the induced map 
\[
H^{k-1}_{\m}(M \otimes^L_R R/t) \to H^{k-1}_{\m}(N\otimes^L_R R/t)
\]
is also injective.
\end{lemma}
\begin{proof}
By localisation at a maximal ideal $\m$, we may assume that $(R,\m)$ is a local ring.
Set $l\coloneqq-k$. By local duality \eqref{eq:local-duality}, we know that the following map: 
\begin{equation} \label{eq:lc-assump}
{\rm Ext}^{l}_R(N, \omega^\kdot_R) \to {\rm Ext}^{l}_R(M, \omega^\kdot_R) 
\end{equation}
 induced by $\phi$ is surjective, and our goal is to show that
\begin{equation} \label{eq:lc-goal}
{\rm Ext}^{l+1}_{R/t}(N \otimes^L_R R/t, \omega^\kdot_{R/t}) \to {\rm Ext}^{l+1}_{R/t}(M \otimes^L_R R/t, \omega^\kdot_{R/t}) 
\end{equation}
is surjective as well.

We have the following identification:
\begin{align}
{\rm Ext}^{l+1}_{R/t}(N \otimes^L_R R/t, \omega^\kdot_{R/t}) &\cong {\rm Ext}^{l+1}_{R}(N \otimes^L_R R/t, \omega^\kdot_{R}) \label{eq:ext-restriction} \\
&\cong  {\rm Ext}^{l}_{R}(N, \omega^\kdot_{R}) \otimes_R R/t, \nonumber
\end{align}
where the first isomorphism is Grothendieck's duality, while the second isomorphism follows from the long exact sequence induced from $0 \to R \xrightarrow{\cdot t} R \to R/t \to 0$:
\[
{\rm Ext}^{l+1}_{R}(N, \omega^\kdot_{R}) \xleftarrow{\cdot t} {\rm Ext}^{l+1}_{R}(N, \omega^\kdot_{R})  \leftarrow {\rm Ext}^{l+1}_{R}(N \otimes^L_R R/t, \omega^\kdot_{R})  \leftarrow {\rm Ext}^{l}_{R}(N, \omega^\kdot_{R}) \xleftarrow{\cdot t} {\rm Ext}^{l}_{R}(N, \omega^\kdot_{R})  
 \]
and the assumption on $t$.

By repeating the same argument for $M$, we get that \eqref{eq:lc-goal} identifies with 
\[
{\rm Ext}^{l}_{R}(N, \omega^\kdot_{R}) \otimes_R R/t \to {\rm Ext}^{l}_{R}(M, \omega^\kdot_{R}) \otimes_R R/t.
\]
This map is surjective by (\ref{eq:lc-assump}) concluding the proof.
\end{proof}

\section{Higher Du Bois singularities}

Following \cite{SVV}, we recall the following definition.
\begin{definition}[See Definition \ref{defintro:-predef}]
Let $X$ be a normal variety defined over $\bC$. We say that $X$ is \emph{pre-$m$-Du Bois} for an integer $m\geq 0$ if
\[
\DDB^i_X \cong \Omega^i_{X,h},
\]
for all $0 \leq i \leq m$.
\end{definition}

We recall from \cite{SVV} how the notion of pre-$m$-Du Bois singularities behave under restriction to a hyperplane. 
\begin{lemma} \label{cor:pre-k-Du Bois-restriction}
Let $X$ be a normal quasi-projective variety over $\bC$ and let $H$ be a general member of a base-point-free linear system associated to a line bundle $L$. Fix an integer $m \geq 0$. Then the following statements hold.
\begin{enumerate}
\item If $X$ is pre-$m$-Du Bois, then $H$ is pre-$m$-Du Bois. 
\item If $H$ is pre-$m$-Du Bois and $L$ is ample, then the support of $\cH^{>0}\DDB^m_X$ is of dimension $0$. 
\end{enumerate}
\end{lemma}
\begin{proof}
Part (1) follows by the same proof as in \cite[Theorem A]{SVV}. The idea of the proof is to run an ascending induction on $m$ using the first triangle in Corollary \ref{cor:restrictionDDB-and-IO}.

As for Part (2), suppose that $H$ is pre-$m$-Du Bois. Then  Corollary \ref{cor:restrictionDDB-and-IO} gives us an exact triangle for $0 \leq i \leq m$:
\[
\Omega^{i-1}_{H,h} \otimes_{\cO_H} \cO_H(-H) \to \DDB^{i}_X\otimes^L_{\cO_X} \cO_H \to \Omega^i_{H,h} \xrightarrow{+1}.
\]  
Hence $\cH^{>0}(\DDB^{i}_X\otimes^L_{\cO_X} \cO_H) = 0$. Since $H$ is general, Lemma \ref{lem:general-res-coh} implies that the support of $\cH^{>0}\DDB^i_X$ is of dimension $0$.
\end{proof}

The following result was obtained in the isolated singularities case in \cite{Friedman-Laza1} and in full generality in \cite{popa2024injectivityvanishingdubois}. This strengthens Steenbrink's vanishing: $\cH^j(\DDB^i_X) = 0$ for $i+j > \dim(X)$.

\begin{proposition}[{\cite[Corollary 3.3]{popa2024injectivityvanishingdubois}}] \label{prop:stronger-Steenbrink}
Let $X$ be a normal connected variety of dimension $d$ defined over $\bC$. Let $m\geq 0$ be an integer such that $X$ is pre-$(m-1)$-Du Bois. Then
\[
\mathcal{H}^{d-m} \DDB_X^m=0.
\]
\end{proposition}
\noindent For the convenience of the reader, we review the proof. In what follows we shall use the following fact. Let $C \in D^b_{\rm coh}(X)$ and let $C^{{\rm an}} \in D^b_{\rm coh}(X^{\rm an})$ be its analytification. Then $\cH^j(C)=0$ if and only if $\cH^j(C^{{\rm an}})=0$. We also refer to Remark \ref{remark:DDB-filtered} for a brief discussion on $\Omega^\kdot_X$.
\begin{proof}
Recall that $\bQ_X$ denotes, by abuse of notation, the sheaf $\bQ_{X^{\rm an}}$ on the analytification $X^{\rm an}$.
Consider the spectral sequence for filtered complexes starting on the first page:
\[
E^1_{i,j} \colon \cH^j(\DDB^i_{X^{\rm an}}) \implies \cH^{i+j}(\DDB^\kdot_{X^{\rm an}}).
\]
Note that $\DDB^\kdot_{X^{\rm an}}  \cong \bC_X$ by the Riemann--Hilbert correspondence. Thus for $i+j>0$:
\begin{equation} \label{eq:ss-end-vanishing}
\cH^{i+j}(\DDB^\kdot_{X^{\rm an}})=0.
\end{equation}
Moreover, since $X$ is pre-$(m-1)$-Du Bois, we have that
\[
\cH^j(\DDB^i_{X^{\rm an}}) = \cH^j(\DDB^i_X)  = 0
\]
for all $j>0$ and $i<m$. Last,
\[
\cH^j(\DDB^i_{X^{\rm an}}) = \cH^j(\DDB^i_{X}) = 0
\]
for $i+j > d$ by Lemma \ref{lem:DDB-IO-coh-amplitude}. Of all this is depicted in the figure below.
\begin{figure*}[h]
\begin{tikzpicture}
  \matrix (m) [matrix of math nodes,
    row sep=1.15em, column sep=0.7em,
    text height=1.15ex, text depth=0.25ex]
  {
    \cH^{d-m} & 0 & 0 & \ldots & 0 & {*} & 0 & 0 & 0 & 0 \\
    \cH^{d-m-1} & 0 & 0 & \ldots & 0 & {*} & {*} & 0 & 0 & 0 \\
    \vdots & \vdots & \vdots & \iddots & 0 & {*} & \vdots & \ddots & 0 & 0 \\
    \cH^1 & 0 & 0 & \ldots & 0 & {*} & {*} & {*} & {*} & 0 \\
    \cH^0 & {*} & {*} & \ldots & {*} & {*} & {*} & {*} & {*} & {*} \\
    & \DDB^0_{X^{\rm an}} & \DDB^1_{X^{\rm an}} & \ldots & \DDB^{m-1}_{X^{\rm an}} & \DDB^m_{X^{\rm an}} & \DDB^{m+1}_{X^{\rm an}} & \ldots & \DDB^{d-1}_{X^{\rm an}} & \DDB^d_{X^{\rm an}}. \\
  };
  
  \draw[->, red] (m-1-6) -- (m-1-7);
  \draw[->, red] (m-1-6) -- (m-2-8);
\end{tikzpicture}
\caption{Spectral sequence}
\end{figure*}

Now, we claim that
\[
E^1_{m,d-m} = \cdots = E^\infty_{m,d-m}.
\]
Indeed, the connecting homomorphisms in the spectral sequence going in and out of $E^\star_{m,d-m}$ are all zero as indicated by the arrows in the figure.

The claim concludes the proof that $\cH^{d-m}(\DDB^m_{X^{\rm an}}) = 0$ given that $E^\infty_{m,d-m}=0$ by \eqref{eq:ss-end-vanishing}. This immediately implies that
\[
\cH^{d-m}(\DDB^m_X) = 0. \qedhere
\]
\end{proof}

With the above, we can give a purely local cohomology definition of pre-$m$-Du Bois singularities. In the isolated case this was proven in the first version of \cite{KW24}. 

\begin{proposition} \label{prop:lc-def-ofkDuBois}
Let $X$ be a normal connected variety of dimension $d$ defined over $\bC$. Then $X$ is pre-$m$-Du Bois for an integer $m\geq 0$ if and only if
\begin{equation*} \label{eq:lc-def-ofkDuBois}
H^j_\m(\Omega^i_{X,h}) \xrightarrow{\phi} H^j_\m(\DDB^i_{X})
\end{equation*}
is injective for every maximal ideal $\m$ and all $i, j \geq 0$ such that $i+j \leq d$ and $i \leq m$.
\end{proposition}
\begin{proof}
Since the question is local we may assume in the proof that $X$ is affine.

The implication from left to right is clear. Hence, we may assume that map $\phi$ is injective and aim for showing that $X$ is pre-$m$-Du Bois. By ascending induction on $m$, we may assume that $X$ is pre-$(m-1)$-Du Bois. Moreover, by ascending induction of dimension, we may assume that the statement of the proposition holds for all normal connected varieties of dimension at most $d-1$.

First, we will show by induction on dimension that $\Omega^i_{X,h}$ and $\DDB^i_{X}$ differ only at a finite collection of points. To this end, let us pick $H$ a general hyperplane section. Consider the following diagram for $0 \leq i \leq m$:
\[
\begin{tikzcd}
\DDB^{i-1}_H \otimes^L_{\cO_H} \cO_H(-H) \ar[r]  & \DDB^{i}_X \otimes^L_{\cO_X} \cO_H \ar[r]& \DDB^i_H \ar[r, "{+1}"]  & \hphantom{a} \\
\Omega^{i-1}_{H,h} \otimes_{\cO_H} \cO_H(-H) \ar[r] \ar[u, "\cong"] & \Omega^{i}_{X,h} \otimes_{\cO_X}  \cO_H  \ar[u] 
 \ar[r] & \Omega^i_{H,h} \ar[r, "{+1}"]  \ar[u] &.
\end{tikzcd}
\]
Here the top row exists and is an exact triangle by Corollary \ref{cor:restrictionDDB-and-IO}. The left vertical arrow is an isomorphism because $H$ is $(m-1)$-pre-Du Bois by Lemma \ref{cor:pre-k-Du Bois-restriction}(1) since so is $X$. Finally, the bottom row is constructed from the top row by applying\footnote{we are using that $H$ is general, so that $\cH^0(\DDB^{i}_X \otimes^L_{\cO_X} \cO_H) \cong \cH^0(\DDB^{i}_X) \otimes_{\cO_X} \cO_H$ by Lemma \ref{lem:general-res-coh}.} $\cH^0$. The bottom row is exact because 
\[
\cH^{>0}(\DDB^{i-1}_H \otimes^L_{\cO_H} \cO_H(-H)) = \cH^{>0}(\Omega^{i-1}_{H,h} \otimes_{\cO_H} \cO_H(-H)) = 0.
\]
By Lemma \ref{lem:lc-restrict-to-divisor}, we know that  $H^j_\m(\Omega^i_{X,h}\otimes_{\cO_X}  \cO_H ) \to H^j_\m(\DDB^i_{X}\otimes^L_{\cO_X}  \cO_H )$ is injective for all $i,j \geq 0$ such that $i+j \leq d-1$ as long as $i \leq m$. Hence, by running a diagram chase in the above diagram, we get that the map
\[
H^j_\m(\Omega^i_{H,h} ) \to H^j_\m(\DDB^i_{H})
\]
is injective as well for $i+j \leq d-1$ and $i \leq m$. In particular, the assumptions of the currently-being-proven proposition are fulfilled for $H$, and so by induction on dimension, $H$ is pre-$m$-Du Bois, that is, $\Omega^i_{H,h} \cong \DDB^i_H$ for all $0 \leq i \leq m$. This shows that \[
\Omega^i_{X,h}  \to  \DDB^i_{X}
\]
is an isomorphism outside of a finite collection of points by Lemma \ref{cor:pre-k-Du Bois-restriction}(2).\\

To conclude the proof, consider the following exact triangle
\[
\Omega^m_{X,h} \to \DDB^m_{X} \to C \xrightarrow{+1},
\]
where $C$ is supported at a finite collection of points as proven in the above paragraph. By Theorem \ref{thm:surj-lc} below, we know that
\[
H^j_\m(\Omega^m_{X,h}) \xrightarrow{\phi} H^j_\m(\DDB^m_{X})
\]
is surjective for every $j \geq 0$ and every maximal ideal $\m$. By assumptions, the above map $\phi$ is injective for $i+j \leq d$, and so it is an isomorphism in this range. In particular, \[
\cH^j(C) = H^j_\m(C) = 0 \quad \text{ for all }\quad  j < d - m.
\]
On the other hand, $C \in D^{<d-m}_{\rm coh}(X)$ by Proposition \ref{prop:stronger-Steenbrink} and Steenbrink's vanishing (Lemma \ref{lem:DDB-IO-coh-amplitude}). This immediately implies that $C=0$.
\end{proof}

Very recently, Kovacs proved that 
\[
H^j_\m(\Omega^i_{X,h}) \xrightarrow{\phi} H^j_\m(\DDB^i_{X})
\]
is always surjective for $i \leq m$ when $X$ is pre-$(m-1)$-Du Bois (\cite{kovács2025complexesdifferentialformssingularities}). However, in the above proposition we only needed a weaker version of this statement. 

\begin{theorem} \label{thm:surj-lc}
Let $X$ be a normal connected variety of dimension $d$ defined over $\bC$. Fix an integer $m \geq 0$. Assume that $X$ is pre-$(m-1)$-Du Bois and that it is pre-$m$-Du Bois outside a finite set of points. Then
\begin{equation} \label{eq:surj-lc-goal}
H^i_\m(\Omega^m_{X,h}) \xrightarrow{\phi} H^i_\m(\DDB^m_{X})
\end{equation}
is surjective for every maximal ideal $\m$ and every $i \geq 0$.
\end{theorem}
\noindent This result has been proven in {\cite[Theorem 5.1]{popa2024injectivityvanishingdubois}} assuming that $X$ is projective. In what follows we generalise their result to the non-projective setting. The reader is recommended to consult Remark \ref{remark:DDB-filtered} before proceeding with the proof. 
\begin{proof}
Fix a non-pre-$m$-Du Bois point $\m \in X$. By compactifying, we may reduce to the case that $X$ is projective and there exists a closed subset $Z \subseteq X$ disjoint from $\m$ such that $U := X \,\backslash\, Z$ is affine and:
\begin{enumerate}
            \item $X$ is pre-$(m-1)$-Du Bois on $U$,
            \item $X$ is pre-$m$-Du Bois on $U \,\backslash\, \{\m\}$.
\end{enumerate}
Here, replacing $X$ by its compactification does not change the local cohomology at $\m$  thanks to excision (see e.g.\ \cite[Chapter 3, Exercise 2.3(f) and 2.5]{Har}). 

The key idea of the proof, as explained below, is to extend $\Omega^m_{U,h}$ to a complex $\cF^m$ on $X$ which agrees with $\DDB^m_X$ along $Z$. To this end, we shall pushforward $\tau^{>0} \DDB^m_U$ to a complex $\cG^m \in D^b_{\rm coh}(X)$ on $X$ and define $\cF^m$ as the cocone of a natural map $\varphi_m \colon \DDB^m_X \to \cG^m$. Due to limitations of triangulated categories, we will need to proceed cautiously to ensure compatibility of this construction with the map $\theta_m \colon \DDB^m_X[-m] \to \DDB^{\leq m}_X$ from Remark \ref{remark:DDB-filtered}, as that will be needed later in the proof. We emphasise that $\cF^m$ and $\cG^m$ are constructed ad hoc and do not belong to any general family.\\


Given that $X$ is pre-$m$-Du Bois on $U \,\backslash\, \{\m\}$, we have that $\tau^{>0} \DDB^m_U \in D^b_{\rm coh}(U)$ is supported at $\m$.
Let $j \colon U \hookrightarrow X$ be the natural open immersion.
Define 
\[
\cG^m :=  Rj_*\tau^{>0} \DDB^m_U.
\]
Consider the spectral sequence for the canonical filtration (\cite[Tag 015J]{stacks-project}):
\[
E^{a,b}_2 = R^aj_*\cH^b(\tau^{>0} \DDB^m_U) \implies \cH^{a+b}(\cG^m).
\]
Since $\cH^b(\tau^{>0} \DDB^m_U)$ is supported at $\m$, we have that  $E^{a,b}_2$ is coherent for $a=0$ and equal to zero for $a>0$. The latter implies that
\[
\cH^i(\cG^m) = j_*\cH^i(\tau^{>0} \DDB^m_U)
\]
for all $i \in \bZ$. In particular, $\cG^m$ belongs to the bounded derived category of \emph{coherent} sheaves $D^{b, >0}_{\rm coh}(X)$ and is supported at $\m$. 

We denote by $\varphi_m$ the composition of the following natural maps:
\begin{equation} \label{eq:varphi}
\varphi_m \colon \DDB^{m}_X \to Rj_*\DDB^m_U \to Rj_*\tau^{>0}\DDB^m_U =\cG^m.
\end{equation}
By construction, the natural map induced by adjunction:
\[
\cG^m|_U = Lj^*Rj_* \tau^{>0}\DDB^m_U \xrightarrow{\cong} \tau^{>0} \DDB^m_U
\]
is a quasi-isomorphism. Hence
we have the following exact triangle in $D^b_{\rm coh}(U)$:
\begin{equation}
\label{eq:UDDB}
\Omega^m_{U,h} \xrightarrow{\rm can} \DDB^m_{U} \xrightarrow{\varphi_m|_U} \cG^m|_U \xrightarrow{+1}.
\end{equation}

\begin{claim} \label{claim:psimap}
There exists a map $\psi_m \colon \DDB^{\leq m}_X \to \cG^m[-m]$ rendering the following diagram
\[
\begin{tikzcd}
    \DDB^{m}_X[-m] \ar{r}{\theta_m} \ar{d}{\varphi_m} & \DDB^{\leq m}_X \ar{d}{\psi_m} \\
\cG^m[-m] \ar{r}{=}  & \cG^m[-m]
\end{tikzcd}
\]
in $D^b(X,\bC)$ commutative. Here, the map $\theta_m$ was fixed in Remark \ref{remark:DDB-filtered}.
\end{claim}
\begin{proof}
Note that $\DDB^i_U \cong \Omega^i_{U,h}$ for $i \leq m-1$ by Assumption (1), and so by tracing through the rightmost column of (\ref{eq:DDBfilteredexact}) we get that:
\[
\DDB^{\leq m-1}_U \in D^{<m}(U, \bC).
\]
Since $\cG^m \cong Rj_*(\cG^m|_U)$ in $D^b_{\rm coh}(X)$, the same isomorphism holds in $D^b(X,\bC)$, and so, using adjunction, we see that for $r \in \{0,1\}$:
\begin{align*}{\rm Hom}_{D^b(X,\bC)}(\DDB^{\leq m-1}_X[-r], \cG^m[-m])&\cong {\rm Hom}_{D^b(U,\bC)}(\DDB^{\leq m-1}_U[-r], \cG^m|_U[-m]) \\
&=0
\end{align*} 
given that $\cG^m|_U[-m] \in D^{>m}_{\rm coh}(U)$. 
Thus, by applying ${\rm Hom}_{D^b(X,\bC)}(-, \cG^m[-m])$ to the last column of (\ref{eq:DDBfilteredexact}) we get an isomorphism:
\begin{equation}\label{eq:phipsi}
{\rm Hom}_{D^b(X,\bC)}(\DDB^{\leq m}_X, \cG^m[-m]) \underset{\theta^*_m}{\xrightarrow{\cong}} {\rm Hom}_{D^b(X,\bC)}(\DDB^{m}_X[-m], \cG^m[-m]).
\end{equation}
The sought-after map $\psi_m$ is then induced by $\varphi_m$ from (\ref{eq:varphi}).
\end{proof}

The 9-lemma (see \cite[Tag 05R0]{stacks-project} or \cite[Proposition 1.1.11]{BBDG18})  applied to the diagram in Claim \ref{claim:psimap}, allow us to extend it to the following diagram with each row and column being an exact triangle in $D^b(X,\bC)$\footnote{we warn the reader that the map $\DDB^{\leq m}_X \to \DDB^{\leq m-1}_X$ in this diagram might not be the same as in  Diagram (\ref{eq:DDBfilteredexact})}:
\begin{equation} \label{eq:BigPhiPsi}
\begin{tikzcd}
\cF^{m}[-m] \ar{r}{\gamma_m} \ar{d}{v_m} & \cF^{\leq m} \ar{d}{u_m} \ar{r} & \cF^{\leq m-1} \ar{d}{\cong} \ar{r}{+1} & \hphantom{a}  \\
\DDB^{m}_X[-m] \ar{r}{\theta_m} \ar{d}{\varphi_m} & \DDB^{\leq m}_X \ar{d}{\psi_m} \ar{r} & \DDB^{\leq m-1}_X  \ar{r}{+1} & \hphantom{a}  \\
\cG^m[-m] \ar{r}{=} \ar{d}{+1} & \cG^m[-m], \ar{d}{+1} \ &  & \hphantom{a}  \\
\hphantom{a} &  \hphantom{a}  & \hphantom{a} 
\end{tikzcd}
\end{equation}
Here $\cF^m$, $\cF^{\leq m}$, and $\cF^{\leq m-1}$ are defined via this diagram.  

We make two observations about $\cF^m$. First, since $\cG^m$ is supported at $\m$, the map \[v_m \colon \cF^m \to \DDB^m_X\] is a quasi-isomorphism after restricting to $X \, \backslash \, \m$.  Second, $\cH^0(\cF^m) = \cH^0(\DDB^m_X) = \Omega^m_{X,h}$ given that $\cG^m \in D^{>0}_{\rm coh}(X)$. In turn, the induced map
\begin{equation} \label{eq:isoonU}
\Omega^m_{X,h} \xrightarrow{\cong} \cH^0(\cF^m) \xrightarrow{{\rm can}} \cF^m
\end{equation}
restricts to a quasi-isomorphism on $U$  by   (\ref{eq:UDDB}).

\begin{claim}
The induced map on hypercohomology:
\[
v_m \colon \bH^i(X, \cF^m) \to \bH^i(X, \DDB^m_X)
\]
is surjective for every $i \in \bZ$. 
\end{claim}
\begin{proof}[Proof of Claim]


Consider the natural map $\bC \cong \DDB^\kdot_X \xrightarrow{\Theta_m} \DDB^{\leq m}_X$ in $D^b(X,\bC)$. Since the composition 
\[
\bC \to \DDB^{\leq m}_X \xrightarrow{\psi_m} \cG^m[-m]
\]
is zero as the last term lives in $D^{>m}(X,\bC)$, we get a factorisation:
\[
\bC \to \cF^{\leq m} \xrightarrow{u_m} \DDB^{\leq m}_X
\]
by tracing through the second column of (\ref{eq:BigPhiPsi}). 
In turn, given that $X$ is projective, the $E_1$-degeneration of the Hodge-to-de Rham spectral sequence for the filtered Deligne-Du Bois complex of $X$ implies that the composition
\[
H^i(X, \bC) \to \bH^i(X, \cF^{\leq m}) \xrightarrow{u_m} \bH^i(X, \DDB^{\leq m}_X)
\]
is surjective for every $i$. In particular, the second map $u_m \colon \bH^i(X, \cF^{\leq m}) \to \bH^i(X, \DDB^{\leq m}_X)$ is surjective for every $i$ as well.

Finally, by tracing through the following diagram:
\[
\begin{tikzcd}
\bH^{i-1}(X,\cF^{\leq m-1}) \ar{d}{\cong} \ar{r} & \bH^i(X, \cF^m[-m]) \ar{r}{\gamma_m} \ar{d}{v_m} & \bH^i(X, \cF^{\leq m}) \ar[two heads]{d}{u_m} \ar{r} & \bH^i(X, \cF^{\leq m-1}) \ar{d}{\cong} \\
 \bH^{i-1}(X, \DDB^{ \leq m-1}_X)  \ar{r} & \bH^i(X, \DDB^m_X[-m]) \ar{r}{\theta_m} & \bH^i(X, \DDB^{\leq m}_X) \ar{r} & \bH^i(X, \DDB^{ \leq m-1}_X).
\end{tikzcd}
\]
we get that the map:
\[
v_m \colon \bH^i(X, \cF^m[-m]) \to \bH^i(X, \DDB^m_X[-m])
\]
is surjective for every $i \in \bZ$ as required. Note that in the above proof we implicitly used the fundamental fact that the cohomology of a coherent sheaf is the same as when it is computed in the category of sheaves of abelian groups.
\end{proof}
Now that we have proven the claim, we get from tracing through the diagram
\[
\begin{tikzcd}
\bH^{i-1}(X \, \backslash \, \m, \cF^m)  \ar{r} \ar{d}{\cong} & \bH^i_{\m}(\cF^m) \ar{r} \ar{d}{v_m} & \bH^i(X, \cF^m) \ar{r} \ar[two heads]{d}{v_m} & \bH^{i}(X \, \backslash \, \m, \cF^m) \ar{d}{\cong}   \\
\bH^{i-1}(X \, \backslash \, \m, \DDB^m_X)  \ar{r} & \bH^i_{\m}(\DDB^m_X) \ar{r} & \bH^i(X , \DDB^m_X) \ar{r} & \bH^{i}(X\, \backslash \, \m, \DDB^m_X)
\end{tikzcd}
\]
that 
\[
v_m \colon H^i_\m(\cF^m) \to H^i_\m(\DDB^m_{X})
\]
is surjective for every $i$. By excision and the fact that (\ref{eq:isoonU}) is an isomorphism on $U$, we have that $H^i_\m(\cF^m) \cong H^i_\m(\Omega^m_{X,h})$, and so the surjectivity of (\ref{eq:surj-lc-goal}) is proven as required.
\end{proof}

\subsection{Higher local cohomology of reflexified cotangent sheaf}
We conclude this section by the proof of Lemma \ref{lem:small-introduction} from introduction.

\begin{proof}[{Proof of Lemma \ref{lem:small-introduction}}]
Let $H$ be a general hyperplane. Since $X$ has rational singularities, so does $H$, and hence we have that 
\[
\Omega^i_{X,h} = \Omega^{[i]}_X \quad \text{ and } \quad \Omega^i_{H,h} = \Omega^{[i]}_H
\]
for all integers $i$ by Proposition \ref{prop:KebekusSchnell}.

Consider the following diagram
\[
\DDB^{m-1}_H \otimes^L_{\cO_H} \cO_H(-H) \to \DDB^{m}_X\otimes^L_{\cO_X}\!\cO_H \to \DDB^m_H \xrightarrow{+1} 
\]
from  Corollary \ref{cor:restrictionDDB-and-IO}. Since $X$ is pre-$(m-1)$-Du Bois, so is $H$ by Lemma \ref{cor:pre-k-Du Bois-restriction} (1). Thus $\cH^{>0}(\DDB^{m-1}_H)=0$, and so by applying $\cH^0$ to the above diagram, we get the following short exact sequence: 
\[
0 \to \Omega^{[m-1]}_H(-H) \to \Omega^{[m]}_X\otimes_{\cO_X}\!\cO_H \to \Omega^{[m]}_H \to 0.
\]
Here we are also using that $\cH^0(\DDB^{m}_X\otimes^L_{\cO_X}\!\cO_H) \cong \cH^0(\DDB^{m}_X) \otimes_{\cO_X}\!\cO_H$, see Lemma \ref{lem:general-res-coh}. In particular, for every closed point $x \in X$ we have that
\[
H^1_x(\Omega^{[m]}_X\otimes_{\cO_X}\!\cO_H)=0,
\]
and so, by local duality (\ref{eq:local-duality}),
\[
\cExt^{-1}( \Omega^{[m]}_X\otimes_{\cO_X}\!\cO_H,\omega^\kdot_X) = 0.
\]
Thanks to the calculation in \eqref{eq:ext-restriction} and the fact that $(-)\otimes^L_{\cO_X} \cO_H$ can be replaced by $(-)\otimes_{\cO_X} \cO_H$ by generality of $H$ (see Lemma \ref{lem:general-res-coh} again), we have that
\[
0 = \cExt^{-1}( \Omega^{[m]}_X\otimes_{\cO_X}\!\cO_H,\omega^\kdot_X) = \cExt^{-1}( \Omega^{[m]}_X\otimes^L_{\cO_X}\!\cO_H,\omega^\kdot_X) = \cExt^{-2}(\Omega^{[m]}_X, \omega^\kdot_X) \otimes_{\cO_X}\!\cO_H. 
\]
Therefore $\cExt^{-2}(\Omega^{[m]}_X, \omega^\kdot_X)$ is of dimension $0$. This concludes the proof.
\end{proof}

\section{Higher rational singularities}

Following \cite{SVV} we recall the following definition.

\begin{definition}
Let $X$ be a normal connected variety defined over $\bC$. We say that $X$ is \emph{pre-$m$-rational} for an integer $m\geq 0$ if
\[
\cH^{>0} \DO^i_X = 0
\]
for all integers $0 \leq i \leq m$.
\end{definition}

Recall from Proposition \ref{prop:dualDDB-explicit} that $\DO^i_X = Rf_*\Omega^i_Y(\log E)$ for a strong resolution $f \colon Y \to X$ with exceptional divisor $E$ if $i < {\rm codim}_X({\rm Sing}(X))$. In particular, $0$-rationality is equivalent to rationality. Thus, in view of Proposition \ref{prop:KebekusSchnell} we get the following.

\begin{lemma} \label{lem:equiv-def-pre-k-rat}
Let $X$ be a normal connected variety defined over $\bC$. Fix an integer $m\geq 0$. If $X$ is pre-$m$-rational, then it has rational singularities and $\cH^0(\DO^i_X) = \Omega^{[i]}_X$ for every $i \in \bZ$.

In particular, $X$ is pre-$m$-rational if and only if $\DO^i_X \cong \Omega^{[i]}_X$ for all $0 \leq i \leq m$.
\end{lemma}

Pre-$m$-rationality imposes strong restrictions on depths of differential forms. The following result is well-known; we provide a proof for the convenience of the reader.

\begin{proposition}\label{prop:weakly-k-rational-depth}
Let $m$ be a non-negative integer and let $X$ be a pre-$m$-rational normal connected variety defined over $\bC$ and of dimension $d$. Then
\[
H^j_\m(\Omega^{[i]}_X) = 0
\]
for every maximal ideal $\m$ and all $i, j \geq 0$ such that $i+j < d$ and $i \leq m$.
\end{proposition}
\begin{proof}
By Lemma \ref{lem:equiv-def-pre-k-rat} and Proposition \ref{lem:loccohdual-and-int} we get that
\[
H^j_\m(\Omega^{[i]}_X) = H^j_\m(\DO^i_X) = 0
\]
for every prime ideal $\m$ and all $i, j \geq 0$ such that $i+j < d$ and $i \leq m$. This concludes the proof. 
\end{proof}

Next, we introduce the following tentative definition. We say that a normal connected variety $X$ over $\bC$ is \emph{pre-$m$-$\IC$-rational} if
\[
\cH^{>0}(\IO^i_X)=0
\]
for all integers $0 \leq i \leq m$. Note that it is proven in \cite{popa2024injectivityvanishingdubois} (see also \cite{DOR25} and \cite{PP24}) that pre-$m$-$\IC$-rationality is equivalent to pre-$m$-rationality. However, to maximise the self-containedness of this article, we will invoke this equivalence until Corollary \ref{cor:lc-def-ofkRational}.

Since $\IO^0_X \cong Rf_*\cO_Y$ for a resolution of singularities $f \colon Y \to X$ (see Proposition \ref{prop:all-properties-of-IO}) we see that $X$ is pre-$0$-$\IC$-rational if and only if it has rational singularities. Thus, in view of Proposition \ref{prop:KebekusSchnell} we get the following.

\begin{lemma} \label{lem:equiv-def-pre-k-IC-rat}
Let $X$ be a normal variety defined over $\bC$. Fix an integer $m\geq 0$. If $X$ is pre-$m$-$\IC$-rational, then it has rational singularities and $\cH^0(\IO^i_X) = \Omega^{[i]}_X$ for every $i \in \bZ$. In particular, $X$ is pre-$m$-$\IC$-rational if and only if $\IO^i_X \cong \Omega^{[i]}_X$ for all $0 \leq i \leq m$.
\end{lemma}

\begin{lemma} \label{lem:prekIC-implies-prekDDB}
Let $X$ be a normal connected variety defined over $\bC$. Fix an integer $m \geq 0$ and assume that $X$ is pre-$m$-$\IC$-rational. Then $X$ is pre-$m$-Du Bois.
\end{lemma}
\begin{proof}
Since $X$ is pre-$m$-$\IC$-rational, it has rational singularities (Lemma \ref{lem:equiv-def-pre-k-IC-rat}), and we have that $\Omega^i_{X,h} = \cH^0(\IO^i_X)  = \Omega^{[i]}_X$ for all $i$ by Proposition \ref{prop:KebekusSchnell}.  Now suppose $0 \leq i \leq m$. By \eqref{eq:fact} we have a factorisation:
\[
\Omega^{[i]}_{X} \to \DDB^i_X \to \IO^i_X.
\]
Since $\IO^i_X \cong \Omega^{[i]}_X$, we get that
\[
H^j_\m(\Omega^{i}_{X,h}) = H^j_\m(\Omega^{[i]}_{X}) \to  H^j_\m(\DDB^i_X)
\]
is injective for all integers $j$. Thus $X$ is pre-$m$-Du Bois by Proposition \ref{prop:lc-def-ofkDuBois}.
\end{proof}

\begin{lemma}[{cf.\ \cite[Theorem A]{SVV}}] \label{cor:pre-k-rational-restriction}
Let $X$ be a normal quasi-projective variety over $\bC$ and let $H$ be a general member of a base-point-free linear system associated to a line bundle $L$. Fix an integer $m \geq 0$. Then the following statements hold.
\begin{enumerate}
\item If $X$ is pre-$m$-$\IC$-rational, then $H$ is pre-$m$-$\IC$-rational.
\item If $H$ is pre-$m$-$\IC$-rational and $L$ is ample, then the support of $\cH^{>0}\IO^m_X$ is of dimension $0$.
\end{enumerate}
\end{lemma}
\begin{proof}
Fix an integer $0\leq i \leq m$. We shall use the following exact triangle from Corollary  \ref{cor:restrictionDDB-and-IO}:
\begin{equation} \label{eq:restriction-again}
\IO^{i-1}_{H} \otimes^L_{\cO_H} \cO_H(-H) \to \IO^{i}_X\otimes^L_{\cO_X} \cO_H \to \IO^i_{H} \xrightarrow{+1}.
\end{equation}  

We prove Part (1) by ascending induction on $m$. Assume that $X$ is pre-$m$-$\IC$-rational. In particular, it is pre-$(m-1)$-$\IC$-rational, and so by induction we may assume that $H$ is pre-$(m-1)$-$\IC$-rational. Thus, we have that
\[
\cH^{>0}(\IO^{m-1}_{H}) = 0 \quad \text{ and } \quad \cH^{>0}(\IO^{m}_X) = 0.
\]
By right exactness of tensor product, the latter implies that $\cH^{>0}(\IO^{m}_X\otimes^L_{\cO_X} \cO_H) = 0$. 
Therefore, by \eqref{eq:restriction-again}, we get $\cH^{>0}(\IO^m_{H}) = 0$ as well, concluding the proof of Part (1).

As for Part (2), suppose that $H$ is pre-$m$-$\IC$-rational. Then
\[
\cH^{>0}(\IO^{i-1}_{H}) = \cH^{>0}(\IO^i_{H}) = 0
\]
for all $i \leq m$, and so by \eqref{eq:restriction-again} we get that:
\[
\cH^{>0}(\IO^{i}_X\otimes^L_{\cO_X} \cO_H) = 0
\]
for all $i \leq m$. Since $H$ is general and $L$ is ample, Lemma \ref{lem:general-res-coh} implies that 
\[
\dim \Supp(\cH^{>0}(\IO^i_X))=0. \qedhere
\]
\end{proof}
In the case of usual pre-$m$-rationality, the above result was obtained in \cite{SVV}.

\subsection{Definition of higher rationality via local cohomology}

First, we recall a key result from \cite{popa2024injectivityvanishingdubois}, and, for the convenience of the reader, append a proof thereof.
\begin{theorem}[{\cite[Theorem 10.5 (Sung Gi Park)]{popa2024injectivityvanishingdubois}}] \label{lem:surj-lc-k-rational}
Let $X$ be a normal connected variety defined over $\bC$ of dimension $d$. Fix an integer $m \geq 0$ and assume that the natural map $\DDB^i_X \to \IO^i_X$ is an isomorphism for all $0 \leq i \leq m-1$. Then the map
\[
H^j_\m(\DDB^m_{X}) \xrightarrow{\phi} H^j_\m(\IO^m_{X})
\]
is surjective for every maximal ideal $\m$ and every $j \geq 0$. Specifically:
\begin{enumerate}
    \item $H^j_\m(\IO^m_X) = 0$ for $j < d-m$,
    \item $\phi$ is a surjection for $j=d-m$, and
    \item $\phi$ is an isomorphism for $j>d-m$.
\end{enumerate}
\end{theorem}
\begin{proof}
Consider the following exact triangle
\[
 C \to \bQ_X[d] \to \IC_X \xrightarrow{+1}
\]
where $C$ is the cocone of the composition  $\bQ_X[d] \to {}^p\cH^0(\bQ_X[d]) \to \IC_X$ from \eqref{eq:trivial-q-leq0} and \eqref{IC2}. By the same reference, the map ${}^p\cH^0(\bQ_X[d]) \to \IC_X$ is surjective, and so \eqref{eq:trivial-q-leq0} and the long exact sequence of perverse cohomology \eqref{eq:les-of-t-cohomology} imply that
\[
C \in {}^pD^{\leq 0}.
\]
The above exact triangle extends to an exact triangle in $D^b_{\rm MHM}(X)$: 
\[
 C^H \to \bQ^H_X[d] \to \IC^H_X  \xrightarrow{+1}
\]
with $C^H \in D^{\leq 0}$ by faithfullness and $t$-exactness of ${\rm rat}$ (see Subsection \ref{ss:HM}(1)(2)).  

Consider the induced exact triangles for every $i$:
\begin{equation} \label{eq:grdr-ses}
\Gr_{-i}\DR(C^H) \to \Gr_{-i}\DR(\bQ^H_X[d]) \to \Gr_{-i}\DR(\IC^H_X) \xrightarrow{+1}.
\end{equation}
By assumptions, the following map is an isomorphism for $i<m$:
\[
\DDB^i_X[d-i] = \Gr_{-i}\DR(\bQ^H_X[d]) \to \Gr_{-i}\DR(\IC^H_X) = \IO^i_X[d-i]
\]
(cf.\ (\ref{eq:GRDRTrivial}) and Definition \ref{def:io}). Therefore:
\begin{alignat*}{3}
\Gr_{>-m}\DR(C^H) &= 0, &&\text{ by \eqref{eq:grdr-ses}, thus:} \\
\Gr_{-m}\DR(C^H) &\in {}^pD^{\leq 0}\quad &&\text{ by $C^H \in D^{\leq 0}$ and  Lemma \ref{lem:grdr-boundary}(2), and so:} \\
R\Gamma_\m(\Gr_{-m}\DR(C^H)) &\in D^{\leq 0} &&\text{  by Remark \ref{rem:perv-coh-maximal-only}}.
\end{alignat*}
By \eqref{eq:grdr-ses}, we have an exact triangle (cf.\ \eqref{eq:GRDRTrivial} and Definition \ref{def:io}):
\[
R\Gamma_\m(\Gr_{-m}\DR(C^H)) \to R\Gamma_\m(\DDB^m_X[d-m]) \to R\Gamma_\m(\IO^m_X[d-m]) \xrightarrow{+1}.
\]

Given that $R\Gamma_\m(\Gr_{-m}\DR(C^H)) \in D^{\leq 0}$, this immediately implies that 
\[
H^j_\m(\DDB^m_{X}) \xrightarrow{\phi} H^j_\m(\IO^m_{X})
\]
is an isomorphism for $j>d-m$ and a surjection for $j=d-m$. Finally, $H^j_\m(\IO^m_{X})=0$ for $j<d-m$ by Lemma \ref{lem:loccohdual-and-int}. This concludes the proof.
\end{proof}

\begin{proposition} \label{prop:lc-def-ofkRational}
Let $X$ be a normal connected variety of dimension $d$ defined over $\bC$. Fix an integer $m \geq 0$. Then $X$ is pre-$m$-$\IC$-rational if and only if
\begin{equation} \label{eq:ass-lc-def-ofkRational}
H^j_\m(\Omega^i_{X,h}) \xrightarrow{\phi} H^j_\m(\IO^i_{X})
\end{equation}
is injective for every maximal ideal $\m$ and all $i, j \geq 0$ such that $i+j \leq d$ and $i \leq m$.
\end{proposition}
\noindent One can also argue by avoiding Theorem \ref{lem:surj-lc-k-rational} and instead mimicking the proof of Proposition \ref{prop:lc-def-ofkDuBois}. However, such a proof is a bit longer and more complicated.
\begin{proof}
First, assume that $X$ is pre-$m$-$\IC$-rational. Then, Lemma \ref{lem:equiv-def-pre-k-IC-rat} implies that $\Omega^i_{X,h} = \Omega^{[i]}_X$  and the map 
\[
\Omega^i_{X,h} = \Omega^{[i]}_{X} \to \IO^i_X
\]
is an isomorphism for all $i \leq m$. Therefore, it is an isomorphism on local cohomology. This concludes the proof from left to right.\\

For the implication from right to left, assume that the map $\phi$ from \eqref{eq:ass-lc-def-ofkRational} is injective in the required range. For $m=0$, we have that $\IO^0_X = Rf_*\cO_Y$ for a resolution of singularities $f \colon Y \to X$ by Proposition \ref{prop:all-properties-of-IO}(1). In particular, $X$ has rational singularities by a standard argument involving local duality \eqref{eq:local-duality} and the Grauert--Riemenschneider vanishing. Thus  
\[
\Omega^i_{X,h} = \cH^0(\IO^i_X) = \Omega^{[i]}_X
\]
for every $i$ by Proposition \ref{prop:KebekusSchnell}. 

Moreover, by ascending induction on $m$ we may assume that $X$ is pre-$(m-1)$-$\IC$-rational, that is: $\IO^i_X \cong \Omega^{[i]}_X$ for $i \leq m-1$ (see  Lemma \ref{lem:equiv-def-pre-k-IC-rat}). Next by the factorisation from \eqref{eq:fact}:
\[
\Omega^i_{X,h} \to \DDB^i_X \to \IO^i_X,
\]
Proposition \ref{prop:lc-def-ofkDuBois}, and the injectivity of \eqref{eq:ass-lc-def-ofkRational} we get that $X$ is pre-$m$-Du Bois. In particular, $\DDB^i_X \cong \Omega^{[i]}_X$ for $i \leq m$. Hence,
\[
\IO^i_X \cong \DDB^i_X.
\]
for all $i \leq m-1$. Therefore, the assumptions of Theorem \ref{lem:surj-lc-k-rational} are satisfied which combined with the injectivity of \eqref{eq:ass-lc-def-ofkRational} imply that
\[
H^j_\m(\Omega^m_{X,h}) \xrightarrow{\phi} H^j_\m(\IO^m_{X})
\]
is an isomorphism for every $j$ and every maximal ideal $\m$. Therefore, by local duality \eqref{eq:local-duality} and Grothendieck duality:
\[
\IO^m_X \cong \Omega^m_{X,h},
\]
and so $X$ is pre-$m$-$\IC$-rational. \qedhere
\end{proof}

Before we proceed with the next corollary, we need the following lemma.
\begin{lemma} \label{lem:injection-fromIO-to-push}
Let $X$ be a normal variety of dimension $d$ defined over $\bC$. Let $f \colon Y \to X$ be any log resolution of singularities with exceptional divisor $E$. Let $D$ be a reduced simple normal crossing divisor containing $E$. Then the map
\[
H^j_\m(\IO^i_X) \xrightarrow{\phi} H^j_\m(Rf_*\Omega^i_Y(\log D))
\]
induced from \eqref{eq:fact} is injective for every maximal ideal $\m$ and all integers $i, j \geq 0$ such that $i+j \leq d$.
\end{lemma}
\begin{proof}
Let $j' \colon U \to Y$ be the inclusion of the complement $U$ of $D$ and let $j \colon U \to X$ be the inclusion of $U$ into $X$. Since $j$ is an open immersion, it is left $t$-exact (see \cite[4.2.4]{BBDG18} or \cite[Remark 3.11(b)]{BMPSTWW2}), and so
\begin{equation} \label{eq:jpush-left-t-exact}
j_*\bQ_U[d] \in {}^pD^{\geq 0}.
\end{equation}
Note that $U$ is smooth. Thus, we can consider the composition 
\[
\psi \colon \IC_X \to {}^pj_*\bQ_U[d] \to j_*\bQ_U[d],
\]
where the first map comes from the Definition  \ref{def:IC} of $\IC_X$. Extend this composite map to the exact triangle
\[
K \to \IC_X \xrightarrow{\psi} j_*\bQ_U[d] \xrightarrow{+1}.
\]
By \eqref{eq:jpush-left-t-exact} and the fact that $\IC_X$ is perverse, we know that $K \in {}^pD^{\geq 0}$. By using the definition of $\IC_X$ or the fact that $\IC_X$ is a simple object in the abelian category of perverse constructible sheaves \eqref{IC3}, we must have that
\[
\IC_X \to {}^p\cH^0(j_* \bQ_U[d])
\]
is an injection in ${\rm Perv}_{\rm cons}(X,\bQ)$, and so in fact $K \in {}^pD^{>0}$. 

Now, the above exact triangle extends automatically, by the same construction, to the exact triangle of Hodge modules:
\[
K^H \to \IC^H_X \xrightarrow{\psi} j_*\bQ^H_U[d] \xrightarrow{+1}.
\]
By the faithfullness and $t$-exactness of ${\rm rat}(-)$ (see Subsection \ref{ss:HM}(1)(2)) we get that $K^H \in D^{> 0}$.

Since the functor $\GrDR$ is left $t$-exact (Proposition \ref{prop:JakubsFavouriteFunctor}(5)), we thus get the following exact triangle
\[
\Gr_\kdot\DR(K^H) \to \Gr_\kdot\DR(\IC^H_X) \xrightarrow{\psi} \Gr_\kdot\DR(j_* \bQ^H_U[d]) \xrightarrow{+1}
\]
with $\Gr_\kdot\DR(K^H) \in {}^pD^{>0}$. This means that $H^j_\m(\Gr_\kdot\DR(K^H)) = 0$ for $j \leq 0$, and so
\begin{equation} \label{eq:injection-fromIO-to-push-key}
H^j_\m(\Gr_\kdot\DR(\IC^H_X)) \to H^j_\m(\Gr_\kdot\DR(j_* \bQ^H_U[d]))
\end{equation}
is injective for all $j \leq 0$.

Now, recall that (Definition \ref{def:io}):
\[
\Gr_\kdot\DR(\IC^H_X) = \IO^0_X[d] \oplus \cdots \oplus \IO^d_X[0].
\]
Moreover, by Remark \ref{remark:jpush}:
\begin{align*}
\Gr_\kdot\DR(j_* \bQ^H_U[d]) \cong Rf_*\Omega^0_Y(\log D)[d] \oplus \cdots \oplus Rf_*\Omega^d_Y(\log D)[0].
\end{align*}
Therefore, \eqref{eq:injection-fromIO-to-push-key} translates to the injectivity of
\[
H^j_\m(\IO^i_X) \xrightarrow{\phi} H^j_\m(Rf_*\Omega^i_Y(\log D))
\]
for every maximal ideal $\m$ and all integers $i, j \geq 0$ such that $i+j \leq d$. This concludes the proof.
\end{proof}

\begin{corollary}
\label{cor:lc-def-ofkRational}
Let $X$ be a normal connected variety of dimension $d$ defined over $\bC$. Let $f \colon Y \to X$ be any log resolution of singularities with exceptional divisor $E$. Fix a reduced simple normal crossing divisor $D$ on $Y$ containing $E$. Then the following conditions are equivalent:
\begin{enumerate}
    \item $X$ is pre-$m$-rational, \vspace{0.2em}
    \item $X$ is pre-$m$-$\IC$-rational,\vspace{0.3em}
    \item $X$ is pre-$m$-Du Bois and $\DDB^i_X \cong \IO^i_X$ for all $0 \leq i \leq m$,\vspace{0.3em}
    \item the map $H^j_\m(\Omega^i_{X,h}) \to H^j_\m(\IO^i_{X})$ is injective for all maximal ideals $\m$, and for all integers $i,j \geq 0$ such that $i+j \leq d$ and $i \leq m$,\vspace{0.3em}
    \item the map $H^j_\m(\Omega^i_{X,h}) \to H^j_\m(\DO^i_{X})$ is injective for all maximal ideals $\m$, and for all integers $i,j \geq 0$ such that $i+j \leq d$ and $i \leq m$,\vspace{0.3em}
    \item the map $H^j_\m(\Omega^i_{X,h}) \to H^j_\m(Rf_*\Omega^i_Y(\log D))$ is injective for all maximal ideals $\m$, and for all integers $i,j \geq 0$ such that $i+j \leq d$ and $i \leq m$.
\end{enumerate}
\end{corollary}
\begin{proof}
By \eqref{eq:fact} we have a factorisation:
\[
H^j_\m(\IO^i_{X}) \to H^j_\m(\DO^i_{X}) \to H^j_\m(Rf_*\Omega^i_Y(\log D)).
\]
The composition is injective for all $i,j \geq 0$ such that $i+j \leq d$ by Lemma \ref{lem:injection-fromIO-to-push}. Therefore, (4), (5), and (6) are all equivalent. Next, (2) and (4) are equivalent by Proposition \ref{prop:lc-def-ofkRational}.\\

Therefore, it remains to show that (1), (2), and (3) are equivalent. Note that (3) immediately implies (2). Hence by Lemma \ref{lem:equiv-def-pre-k-rat} and Lemma \ref{lem:equiv-def-pre-k-IC-rat} we may assume from now on that $X$ is rational and
\[
\Omega^i_{X,h} = \cH^0(\IO^i_X) = \cH^0(\DO^i_X) = \Omega^{[i]}_X
\]
for all $i$ (see Proposition \ref{prop:KebekusSchnell}).

Next, we show that (2) and (3) are equivalent. As mentioned above, (3) implies (2). Now assume (2), that is, $X$ is pre-$m$-$\IC$-rational. This implies that $X$ is pre-$m$-Du Bois by Lemma \ref{lem:prekIC-implies-prekDDB} and in view of $\DDB^i_X \cong \Omega^{[i]}_X$ and $\IO^i_X \cong \Omega^{[i]}_X$, we get that $\IO^i_X \cong \DDB^i_X$ for all $0 \leq i \leq m$. Thus (2) implies (3).

We conclude by showing that (1) and (3) are equivalent. First, (1) immediately implies (5), which is equivalent to (3). As for the opposite direction, suppose that $\DDB^i_X \cong \Omega^{[i]}_X$ and $\IO^i_X \cong \DDB^i_X$ for all $i \leq m$. By Remark \ref{rem:fact} (cf.\ Proposition \ref{prop:all-properties-of-IO}(5)), we have the following factorisation:
\[
\DDB^i_X \xrightarrow{\psi} \IO^i_X \cong \bD(\IO^{d-i}_X[d]) \xrightarrow{\bD(\psi)} \bD(\DDB^{d-i}_X[d]) = \DO^i_X.
\]
By \cite[Proposition 7.4]{PP24}, we get that $\DDB^i_X \cong \Omega^{[i]}_X$ and $\IO^i_X \cong \DDB^i_X$ for all $d-m-1 \leq i \leq d$. In particular, the whole composition $\DDB^i_X \to \DO^i_X$ is an isomorphism as well for $0 \leq i \leq m$. Hence, (3) implies (1) and the proof is concluded. \qedhere
\end{proof}
The use of \cite[Proposition 7.4]{PP24} (which dates back to the argument from \cite{DOR25}) is essential here. We do not know how to prove that pre-$m$-rationality and pre-$m$-IC-rationality are equivalent without the use of weight filtration.

\section{Inversion of adjunction for higher rational singularities}
In this section, we work with the following notation.

\begin{setting} \label{setting:inv-k-rational}
Let $X$ be a normal connected variety of dimension $d$ defined over $\bC$. Let $D$ be a normal Cartier prime divisor on $X$ and let $f \colon Y \to X$ be a log resolution of singularities of $(X,D)$. We denote the exceptional divisor by $E$, the strict transform of $D$ by $G$, and the intersection $E \cap G$ by $F$. 
\[
\begin{tikzcd}
\mathllap{F \coloneqq }\ G \cap E \ar[hook]{r} & G \cup E \ar[hook]{r} \ar{d} & Y \ar{d}\\
& D \ar[hook]{r} & X.
\end{tikzcd}
\]
\end{setting}

\begin{lemma} \label{lem:extension-for-log-pairs}
With notation as in Setting \ref{setting:inv-k-rational}, suppose that $X$ and $D$ have rational singularities. Then
\begin{align*}
f_*\Omega^i_Y(\log G + E) &= \Omega^{[i]}_X(\log D), \\
f_*\Omega^i_Y(\log E) &= \Omega^{[i]}_X =\Omega^i_{X,h}, \text{ and }\\
f_*\Omega^i_G(\log F) &= \Omega^{[i]}_D = \Omega^i_{D,h}
\end{align*}
for all $0 \leq i \leq d$.
\end{lemma}
\begin{proof}
Since $X$ and $D$ are rational, they are Du Bois, and so $(X,D)$ is a Du Bois pair (\cite[Proposition 6.15]{Kol13}). Hence, the first statement follows immediately from \cite[Theorem 4.1]{Graf-Kovacs}. The second statement follows from Proposition \ref{prop:KebekusSchnell} as $X$ and $D$ have rational singularities.
\end{proof}

\begin{lemma} \label{lem:advance-sess-reflexified}
Let $X$ be a normal connected variety of dimension $d$ defined over $\bC$, let $m>0$ be a fixed integer, and let $D$ be a prime normal Cartier divisor. Suppose that 
\begin{equation} \label{eq:advance-ass-supp}
{\rm Supp}\, \cExt^{-2}(\Omega^{[m]}_X, \omega^\kdot_X) \subseteq D
\end{equation}
and for all prime ideals $\p \in {\rm Sing}(D)$ we have that
\[
H^2_\p(\Omega^{[m]}_D) = H^2_\p(\Omega^{[m-1]}_D) = 0.
\]
Then there exist the following short exact sequences which generically agree with those from Lemma \ref{lem:basic-ses}:
\begin{align*}
0 \to \Omega^{[m]}_X(\log D)(-D) \to \Omega^{[m]}_X \to \Omega^{[m]}_D \to 0 \\
0 \to \Omega^{[m]}_X \to \Omega^{[m]}_X(\log D) \to \Omega^{[m-1]}_D \to 0.
\end{align*}
\end{lemma}
\noindent This result is surprisingly subtle. The only way that to prove this lemma that we could come up with, with no extraneous assumptions on the ambient space, such as $H^2_{\m}(\Omega^{[m]}_X) = 0$, is by running a similar argument to that of the proof of the main theorem (Theorem \ref{thm:inversion-of-adjunction}).
\begin{proof}
Let $j \colon U \hookrightarrow X$ be an open subset with complement $Z = X\, \backslash\, U$ such that $U$ and $D|_U$ are smooth, ${\rm codim}_X Z \geq 2$, and ${\rm codim}_D(Z \cap D) \geq 2$. Indeed, we can take $U$ as the complement of ${\rm Sing}(D)\cup {\rm Sing}(X)$ since ${\rm Sing}(X)\cap D\subseteq {\rm Sing}(D)$ by inversion of adjunction for regularity. By pushing forward the short exact sequences from Lemma \ref{lem:basic-ses} via $j$, we get the following exact sequences:
\begin{align}
0 \to \Omega^{[m]}_X(\log D)(-D) \to \Omega^{[m]}_X \xrightarrow{\phi} \Omega^{[m]}_D \label{eq:key-ses-building} \\
0 \to \Omega^{[m]}_X \to \Omega^{[m]}_X(\log D) \xrightarrow{\psi} \Omega^{[m-1]}_D. \nonumber
\end{align}
We denote $\cF \coloneqq {\rm im}(\phi)$ and $\cG \coloneqq {\rm im}(\psi)$. We have short exact sequences:
\begin{align}
&0 \to \cF \to \Omega^{[m]}_D \to \Omega^{[m]}_D/\cF \to 0 \label{eq:first-quotient} \\
&0 \to \cG \to \Omega^{[m-1]}_D \to \Omega^{[m-1]}_D/\cG \to 0. \nonumber
\end{align}
In order to conclude the proof, we need to show that 
\begin{equation} \label{eq:claimcG}
\cF = \Omega^{[m]}_D \quad \text{ and } \quad \cG = \Omega^{[m-1]}_D.
\end{equation}
By contradiction, assume that at least one of these statements is false. Then, in the set of all prime ideals $\p$ such that $\cF_{\p} \neq \Omega^{[m]}_{D,\p}$ or $\cG_{\p} \neq \Omega^{[m-1]}_{D,\p}$, let us pick a minimal such $\p$ with respect to containment. 
By Lemma \ref{lem:basic-ses} and inversion of adjunction for regularity, we have that $\p \in {\rm Sing}(D)$.

By minimality, the localisations of $\Omega^{[m]}_D/\cF$ and $\Omega^{[m-1]}_D/\cG$ at $\p$ are empty or zero-dimensional.
In particular,
\[
H^{i}_\p(\Omega^{[m]}_D/\cF) = H^{i}_\p(\Omega^{[m-1]}_D/\cG) = 0,
\]
for all $i>0$,
and so the long exact sequences of local cohomology applied to \eqref{eq:first-quotient} imply:
\[
H^2_\p(\cF) = H^2_\p(\Omega^{[m]}_D) = 0 \quad \text{ and } \quad H^2_\p(\cG) = H^2_\p(\Omega^{[m-1]}_D) = 0.
\]
By \eqref{eq:key-ses-building}, we thus get that both
\begin{equation} \label{eq:twosurj}
H^2_\p(\Omega^{[m]}_X) \to H^2_\p(\Omega^{[m]}_X(\log D)) \quad \text{ and } \quad H^2_\p(\Omega^{[m]}_X(\log D)(-D)) \to H^2_\p(\Omega^{[m]}_X)
\end{equation}
are surjective. Therefore, the following composition
\begin{equation} \label{eq:comp0-t}
H^2_\p(\Omega^{[m]}_X(-D)) \to H^2_\p(\Omega^{[m]}_X(\log D)(-D)) \to H^2_\p(\Omega^{[m]}_X),
\end{equation}
induced by the natural inclusion  $\Omega^{[k]}_X(-D) \hookrightarrow \Omega^{[m]}_X$, is surjective.

\begin{claim} $H^2_\p(\Omega^{[m]}_X)=0$.
\end{claim}
\begin{proof}
We may assume that $X=\Spec R$ is affine and $D = (t)$ for $t \in R$. Then the surjectivity of \eqref{eq:comp0-t} means that multiplication by $t$:
\[
H^2_\p(\Omega^{[m]}_X) \xrightarrow{\cdot t} H^2_\p(\Omega^{[m]}_X)
\]
is surjective. Via local duality \eqref{eq:local-duality} this translates to the injectivity of:
\[
\cExt^{-2}(\Omega^{[m]}_X, \omega^\kdot_X)_\p \xrightarrow{\cdot t} \cExt^{-2}(\Omega^{[m]}_X, \omega^\kdot_X)_\p.
\]
In view of assumption \eqref{eq:advance-ass-supp}, this is only possible when $\cExt^{-2}(\Omega^{[m]}_X, \omega^\kdot_X)_\p=0$. 
Here, recall that $\mathrm{Supp}(\cExt^{-2}(\Omega^{[m]}_X, \omega^\kdot_X)_\p)=V(\mathrm{ann}(\cExt^{-2}(\Omega^{[m]}_X, \omega^\kdot_X)_\p))$. 
By local duality \eqref{eq:local-duality}, this translates to the vanishing of $H^2_\p(\Omega^{[m]}_X)$.
\end{proof}
In particular, $H^2_{\p}(\Omega^{[m]}_X(\log D))=0$
given that
\[
H^2_{\p}(\Omega^{[m]}_X)\to H^2_{\p}(\Omega^{[m]}_X(\log D))
\]
is surjective as we have seen before.

First, we show that $\cG$ and $\Omega^{[m-1]}_D$ agree at $\p$. To this end, by the long exact sequence of local cohomology applied to the second sequence in \eqref{eq:key-ses-building}, we get an exact sequence 
\[
H^1_\p(\Omega^{[m]}_X(\log D))\to H^1_{\p}(\cG)\to H^2_{\p}(\Omega^{[m]}_X)=0.
\]
Since $\Omega^{[m]}_X(\log D)$ is reflexive and $\p$ has codimension at least two in $X$, we have $H^1_\p(\Omega^{[m]}_X(\log D))=0$. Thus, $H^1_\p(\cG)=0$.
In turn, from the second sequence of \eqref{eq:first-quotient} and $H^0_\p(\Omega^{[m]}_D)=0$, we get that:
\[
(\Omega^{[m-1]}_D/\cG)_{\p}\cong H^0_\p(\Omega^{[m-1]}_D/\cG)=0,
\]
which implies that $\cG$ and $\Omega^{[m-1]}_D$ agree at $\p$. Here the first isomorphism holds, because $\Omega^{[m-1]}_D/\cG$ is supported at $V(\p)$. 

Next, we show that $\cF$ and $\Omega^{[m]}_D$ agree at $\p$. To this end, by the first sequence in \eqref{eq:key-ses-building}, we have an exact sequence
\[
H^1_\p(\Omega^{[m]}_X)\to H^1_\p(\cF)\to H^2_{\p}(\Omega^{[m]}_X(\log D)(-D)).
\]
Since $\Omega^{[m]}_X$ is reflexive, $H^1_\p(\Omega^{[m]}_X)=0$.
Moreover, 
\[
H^2_{\p}(\Omega^{[m]}_X(\log D)(-D))\cong H^2_{\p}(\Omega^{[m]}_X(\log D))=0
\]
since $D$ is Cartier.
Thus $H^1_\p(\cF)=0$.
In particular, by the first sequence of \eqref{eq:first-quotient},
\[
(\Omega^{[m]}_D/\cF)_{\p}\cong H^0_\p(\Omega^{[m]}_D/\cF)=0.
\]
Hence, $\cF$ and $\Omega^{[m]}_D$ agree at $\p$ as well. This and the above paragraph contradict the fact that $\cF_{\p} \neq \Omega^{[m]}_{D,\p}$ or $\cG_{\p} \neq \Omega^{[m-1]}_{D,\p}$ concluding the proof of the lemma.
\end{proof}

\begin{corollary} \label{cor:advanceses}
In Setting \ref{setting:inv-k-rational}, suppose that the assumptions of Lemma \ref{lem:advance-sess-reflexified} hold.
Furthermore, suppose in addition that $X$ and $D$ have rational singularities. Then the following diagrams exist:
    \begin{equation} \label{eq:k-pure-diagram-1-rational}
\begin{tikzcd}[column sep = small]
 Rf_*\Omega^m_Y(\log G + E)(-G)  \arrow[r] & Rf_*\Omega^m_Y(\log E) \arrow[r] & Rf_*\Omega^m_G(\log F) \arrow[r,"+1"] & \hphantom{a} \\
    \Omega^{[m]}_X(\log D)(-D) \arrow[u] \arrow[r] & \Omega^{[m]}_X \arrow[u] \arrow[r] & \Omega^{[m]}_D \arrow[u] \arrow[r,"+1"] & \hphantom{a} 
\end{tikzcd}
\end{equation}
\begin{equation} \label{eq:k-pure-diagram-2-rational}
\begin{tikzcd}[column sep = small]
     Rf_*\Omega^m_Y(\log E) \arrow[r] & Rf_*\Omega^m_Y(\log G + E)  \arrow[r] & Rf_*\Omega^{m-1}_G(\log F)    \arrow[r, "+1"] & \hphantom{a} \\
    \Omega^{[m]}_X \arrow[u] \arrow[r] & \Omega^{[m]}_X(\log D) \arrow[u] \arrow[r] & \Omega^{[m-1]}_D \arrow[u] \arrow[r,"+1"] & \hphantom{a}.
\end{tikzcd}
\end{equation}
\end{corollary}
\begin{proof}
We start with the first diagram. The top row exists by Lemma \ref{lem:basic-ses}. The bottom row is constructed by applying $\cH^0$ to the top row. Specifically, 
\[
f_*\Omega^m_Y(\log E) = \Omega^{[m]}_X \quad \text{ and } \quad  f_*\Omega^m_G(\log F) = \Omega^{[m]}_D
\]
by Lemma \ref{lem:extension-for-log-pairs}, and so
\[
f_*\Omega^m_Y(\log G + E)(-G) = {\rm ker}(\Omega^{[m]}_X \to  \Omega^{[m]}_D) = \Omega^{[m]}_X(\log D)(-D). 
\]
Finally, the map $\Omega^{[m]}_X \to \Omega^{[m]}_D$ is surjective by Lemma \ref{lem:advance-sess-reflexified}. This concludes the construction of the first diagram.

We move to the second diagram. The top row thereof exists by Lemma \ref{lem:basic-ses}. The bottom row is constructed by applying $\cH^0$ to the top row. Specifically
\[
f_*\Omega^m_Y(\log E) = \Omega^{[m]}_X \ \text{ and } \ f_*\Omega^m_Y(\log G + E) = \Omega^{[m]}_X(\log D) \ \text{ and } \ f_*\Omega^{m-1}_G(\log F) = \Omega^{[m-1]}_D
\]
by Lemma \ref{lem:extension-for-log-pairs}. Finally, the map $\Omega^{[m]}_X(\log D) \to \Omega^{[m-1]}_D$ is surjective by Lemma \ref{lem:advance-sess-reflexified}. This concludes the construction of the second diagram.
\end{proof}

\begin{theorem}\label{thm:inversion-of-adjunction}
With notation as in Setting \ref{setting:inv-k-rational},  fix an integer
\[
0 < m < {\rm codim}_D({\rm Sing}(D)) - 2.
\]
Assume that $D$ is pre-$m$-rational. Then $X$ is pre-$m$-rational along $D$.  
\end{theorem}
\begin{proof}
In what follows, we may assume that $X=\Spec R$ is affine and that $D$ is given by a single function $t \in R$.

In this paragraph, we verify that the assumptions of Corollary \ref{cor:advanceses} are all satisfied. First, since $D$ is pre-$m$-rational, it has rational singularities (Lemma \ref{lem:equiv-def-pre-k-rat}). In particular, by inversion of adjunction for rational singularities (\cite{elkik78}) we have that $X$ has rational singularities along $D$. By replacing $X$ by an open subset containing $D$, we may assume that $X$ has rational singularities. Moreover, by ascending induction on $m$, we may assume that $X$ is pre-$(m-1)$-rational. Thus, by Lemma \ref{lem:small-introduction}:
\[
\dim \Supp \cExt^{-2}(\Omega^{[m]}_X, \omega^\kdot_X) = 0,
\]
and so we may assume that $\Supp \cExt^{-2}(\Omega^{[m]}_X, \omega^\kdot_X) \subseteq D$. Moreover, by Proposition \ref{prop:weakly-k-rational-depth} and Lemma \ref{lem:lc-max-to-prime}, we have that
\[
H^j_\p(\Omega^{[m]}_D) = 0
\]
for all prime ideals $\p \in D$ and all $i,j \geq 0$ such that $i+j < \dim \cO_{D,\p}$ and $i \leq m$. Since $m+2 < {\rm codim}_D({\rm Sing}(D))$ we get that
\[
H^2_\p(\Omega^{[i]}_D) = 0
\]
for all $0 \leq i \leq m$ and all prime ideals $\p \in {\rm Sing}(D)$. This concludes the verification that the assumptions of Corollary \ref{cor:advanceses} are all satisfied.\\

Pick a maximal ideal $\m \in D$. Pick integers $j \geq 0$ and $0 \leq i \leq m$ such that $i+j \leq d$. Pick a non-zero cohomology class $\alpha \in H^j_\m(\Omega^{[i]}_X)$. Since local cohomology is Artinian, up to multiplying $\alpha$ by some big power of $t$ we may assume that $t\alpha = 0$, but $\alpha$ is still non-zero. By Corollary \ref{cor:lc-def-ofkRational}(6), to conclude the proof we need to show that $f^*(\alpha) \neq 0$, where:
\[
f^* \colon H^j_\m(\Omega^{[i]}_X) \to H^j_\m(Rf_*\Omega^i_Y(\log E)).
\]

In what follows, given a complex $K$ we denote $K(-D) \coloneqq K \otimes_{\cO_X} \cO_X(-D)$. Since $\cO_X(-D)\cong \cO_X$, we can reformulate the above set-up as saying that we have a non-zero class $\alpha \in H^j_\m(\Omega^{[i]}_X(-D))$ such that  
\begin{equation} \label{eq:iszero}
    \text{ the image of $\alpha$ under the inclusion } H^j_\m(\Omega^{[i]}_X(-D)) \to H^j_\m(\Omega^{[i]}_X) \text{ is zero,}
\end{equation} and our goal is to show that:
\vspace{0.4em}
\begin{equation} \label{eq:inv-k-rational-goal1}
f^*(\alpha) \neq 0 \quad \text{ for } \quad f^* \colon H^j_\m(\Omega^{[i]}_X(-D)) \to H^j_\m(Rf_*\Omega^i_Y(\log E)(-f^*D)).
\end{equation}
\vspace{-0.4em}

\noindent Diagram (\ref{eq:k-pure-diagram-2-rational}) tensored by $\cO_X(-D)$ induces the following diagram:
\[
\begin{tikzcd}[column sep = small]
      & H^{j}_\m(Rf_*\Omega^i_Y(\log E)(-f^*D))  \arrow[r] & H^{j}_\m(Rf_*\Omega^i_Y(\log G+E)(-f^*D)) \\
    0 =  H^{j-1}_\m(\Omega^{[i-1]}_D(-D))\!\!\!\!\!\!\!\!\!\!\! 
 \arrow[r, shorten <= 1.7em] & H^{j}_\m(\Omega^{[i]}_X(-D)) \arrow[u, "f^*"] \arrow[r, "\psi"] & H^{j}_\m(\Omega^{[i]}_X(\log D)(-D)), \arrow[u, "f^*"] 
\end{tikzcd}
\]
where $H^{j-1}_\m(\Omega^{[i-1]}_D)=0$ by Proposition \ref{prop:weakly-k-rational-depth} given that $(i-1) + (j-1) < d-1$ and $D$ is pre-$(m-1)$-rational. Because of this vanishing, $\beta \coloneqq \psi(\alpha)$ is non-zero. Thus, in view of the above diagram, to fulfil our goal \eqref{eq:inv-k-rational-goal1}, it is enough to show that
\vspace{0.4em}
\begin{equation} \label{eq:inv-k-rational-goal2}
f^*(\beta) \neq 0 \ \text{ for } \ f^* \colon H^j_\m(\Omega^{[i]}_X(\log D)(-D)) \to H^j_\m(R\pi_*\Omega^i_Y(\log G+E)(-f^*D)).
\end{equation}
\vspace{-0.4em}

\noindent Next, diagram (\ref{eq:k-pure-diagram-1-rational}) induces the following diagram on local cohomology:
\[
\begin{tikzcd}
    H^{j-1}_\m(Rf_*\Omega^i_G(\log F))  \arrow[r, hook, "(\dagger\dagger)"] & H^{j}_\m(Rf_*\Omega^i_Y(\log G+E)(-G))  \arrow[r] & H^{j}_\m(Rf_*\Omega^i_Y(\log E)) \\
      H^{j-1}_\m(\Omega^{[i]}_D) \arrow[u, hook, "(\dagger)"] \arrow[r, "\theta"] & H^{j}_\m(\Omega^{[i]}_X(\log D)(-D)) \arrow[u, "f^*"] \arrow[r, "\phi"] & H^{j}_\m(\Omega^{[i]}_X), \arrow[u] 
\end{tikzcd}
\]
where the arrow $(\dagger\dagger)$ is injective, because of the vanishing $H^{j-1}_\m(Rf_*\Omega^i_Y(\log E))=0$ coming from Proposition \ref{prop:vanishing}. Moreover, $(\dagger)$ is injective as $D$ is pre-$m$-rational (see Corollary \ref{cor:lc-def-ofkRational}(6)). \\

We have a factorisation:
\[
f^* \colon \Omega^{[i]}_X(\log D)(-D) \to Rf_*\Omega^i_Y(\log G+E)(-f^*D) \to  Rf_*\Omega^i_Y(\log G+E)(-G).
\]
Therefore, to fulfil our goal \eqref{eq:inv-k-rational-goal2}, it is enough to show that
\vspace{0.4em}
\begin{equation} \label{eq:inv-k-rational-goal3}
f^*(\beta) \neq 0 \ \text{ for } \ f^* \colon H^j_\m(\Omega^{[i]}_X(\log D)(-D)) \to H^j_\m(Rf_*\Omega^i_Y(\log G+E)(-G)).
\end{equation}
\vspace{-0.4em}

\noindent Now $\phi(\beta)=\phi(\psi(\alpha)) = 0 \in H^j_\m(\Omega^{[i]}_X)$ by \eqref{eq:iszero} given that the composition
\[
H^j_\m(\Omega^{[i]}_X(-D)) \xrightarrow{\psi} H^j_\m(\Omega^{[i]}_X(\log D)(-D)) \xrightarrow{\phi} H^j_\m(\Omega^{[i]}_X),
\]
is equal to the natural map  $H^j_\m(\Omega^{[i]}_X(-D)) \to H^j_\m(\Omega^{[i]}_X)$. Hence there exists a class 
\[
0\neq \gamma \in H^{j-1}_\m(\Omega^{[i]}_D)
\]
such that $\beta = \theta(\gamma)$. Therefore, our goal \eqref{eq:inv-k-rational-goal3} is fulfilled, given that both $(\dagger)$ and $(\dagger\dagger)$ are injective. \qedhere
\end{proof}

We conclude the article by providing a proof of Corollary \ref{corintro:fibres}.
\begin{proof}[Proof of Corollary \ref{corintro:fibres}]
By Theorem \ref{intro:main-thm}, there exists an open subset $U\subseteq X$ containing $X_s$ such that $U$ is $m$-rational. Since $f$ is proper, we may assume that $U = f^{-1}(V)$ for some open subset $V \subseteq C$. Then by Corollary \ref{cor:pre-k-rational-restriction}(1), we have that every fibre $X_s$ of $f \colon X \to C$ for a general element $s \in V$ is $m$-rational as required. 
\end{proof} 

 \bibliographystyle{amsalpha}
 \bibliography{bibliography}

\end{document}